\documentclass[11pt]{article}
\usepackage[pagebackref,colorlinks=true,urlcolor=blue,linkcolor=red,citecolor=red]{hyperref}
\usepackage{epsfig, graphics, graphicx}
\usepackage{subcaption, comment, bm}
\usepackage[percent]{overpic}
\usepackage{float}
\usepackage{amssymb,bm}
\usepackage{amsthm}
\usepackage{color}
\usepackage{amsmath}
\usepackage[]{amsfonts}
\usepackage{fancyhdr}
\usepackage[]{graphicx,wrapfig}
\usepackage{enumitem}
\usepackage{array}

\newcommand{\tblue}[1]{{\color{blue}{#1}}}

\newtheorem{theorem}{Theorem}[]

\newcommand{\cout}[1]{}

\newcommand{\x}{x}

\usepackage{graphicx}
\usepackage[dvipsnames]{xcolor}

\newcommand{\D}{\mathrm{d}}

\newcommand{\Lc}{\mathcal{L}}
\newcommand{\Mc}{\mathcal{M}}

\newcommand{\Pc}{\mathcal{P}}
\newcommand{\Qc}{\mathcal{Q}}
\newcommand{\Rc}{\mathcal{R}}
\newcommand{\Sc}{\mathcal{S}}
\newcommand{\Tc}{\mathcal{T}}

\newcommand{\Vc}{\mathcal{V}}

\newcommand{\Xc}{\mathcal{X}}

\newcommand{\Zc}{\mathcal{Z}}

\newcommand{\Nb}{\mathbb{N}}

\newcommand{\Rb}{\mathbb{R}}
\newcommand{\Sb}{\mathbb{S}}

\newcommand{\vev}{\textbf{\textit{e}}}
\newcommand{\vf}{\textbf{\textit{f}}}
\newcommand{\vg}{\textbf{\textit{g}}}
\newcommand{\vx}{\textbf{\textit{x}}}
\newcommand{\vy}{\textbf{\textit{y}}}

\newcommand{\vu}{\textbf{\textit{u}}}
\newcommand{\vv}{\textbf{\textit{v}}}

\newcommand{\vgamma}{\pmb{\gamma}}

\newcommand{\vxi}{\pmb{\xi}}

\usepackage{amsthm}
\usepackage{amsfonts}
\newcommand{\Beq}{\begin{equation}}
\newcommand{\Eeq}{\end{equation}}
\newcommand{\beq}{\begin{equation*}}
\newcommand{\eeq}{\end{equation*}}
\newcommand{\bal}{\begin{align}}
\newcommand{\eal}{\end{align}}

\graphicspath{{./Figs/}}

\newtheorem{defn}{Definition}
\newtheorem{rem}{Remark}
\newtheorem{cor}{Corollary}
\newtheorem{lem}{Lemma}

\title{\vspace{-1cm} V-line 2-tensor tomography in the plane}

\author{Gaik Ambartsoumian\thanks{Department of Mathematics, University of Texas at Arlington, Arlington, TX, United States of America.  \url{gambarts@uta.edu}}\and Rohit Kumar Mishra\thanks{Mathematics Discipline, Indian Institute of Technology, Gandhinagar, Gujarat, India. \url{rohit.m@iitgn.ac.in}}
\and Indrani Zamindar\thanks{Mathematics Discipline, Indian Institute of Technology, Gandhinagar, Gujarat, India. \url{indranizamindar@iitgn.ac.in}} }

\begin{document}
\date{}
\maketitle
\begin{abstract}
In this article, we introduce and study various V-line transforms (VLTs) defined on symmetric $2$-tensor fields in $\mathbb{R}^2$. The operators of interest include the longitudinal, transverse, and mixed VLTs, their integral moments, and the star transform. With the exception of the star transform, all these operators are natural generalizations to the broken-ray trajectories of the corresponding well-studied concepts defined for straight-line paths of integration. We characterize the kernels of the VLTs and derive exact formulas for reconstruction of tensor fields from various combinations of these transforms. The star transform on tensor fields is an extension of the corresponding concepts that have been previously studied on vector fields and scalar fields (functions). We describe all injective configurations of the star transform on symmetric 2-tensor fields and derive an exact, closed-form inversion formula for that operator.
\end{abstract}
\vspace{-5mm}

\section{Introduction}\label{Introduction}
The inverse problems of recovering a tensor field in $\Rb^n$ (or a Riemannian manifold) from various combinations of integral operators, such as the longitudinal, transverse, and momentum ray transforms applied to that field, have been studied extensively by various authors. In two-dimensional spaces (for both Euclidean and non-Euclidean settings) these questions have been addressed in the works \cite{derevtsov3,krishnan2019solenoidal, Monard2, Paternain_Salo_Uhlmann_2014, SADIQ2016} and in some of the references mentioned there. Numerous other studies have examined the injectivity, invertibility, and range characterization of the longitudinal ray transform and its moments in higher-dimensional Euclidean spaces (see \cite{krishnan2019momentum,krishnan2020momentum,Louis_2022,Rohit_Suman_2021,Rohit_Suman_2022,Sharafutdinov_Book} and the references therein). Support theorems for integral moments of longitudinal geodesic ray transform and transverse geodesic ray transforms for symmetric $m$-tensor field defined on a Riemannian manifold have been studied in \cite{abhishek2020support, abhishek2019support}.
The reconstruction problem for both the longitudinal and transverse ray transforms in $n$-dimensions ($n \geq 3$) is overdetermined, therefore one would like to have a reconstruction algorithm using only an $n$-dimensional restriction of the data. Such questions related to  restricted data are also well-studied in different settings; for instance, see \cite{Denisiuk_2023,Denisjuk_Paper,krishnan2018microlocal,VRS,Vertgeim2000} and references therein.
These problems, sometimes collectively called tomography of tensor fields, have applications in plasma physics, polarization imaging, prediction of earthquakes, and many other areas related to the propagation of radiation or waves through anisotropic media. 

The transforms studied in the references mentioned above map a tensor field to various families of its integrals along straight lines (or geodesics). Physically, the operators integrating a tensor field (alternatively, a vector field or a function) along straight lines or rays often represent the ballistic regime of particle transport, describing the unimpeded flow of charged particles. At the same time, the mesoscopic regime of particle transport, allowing single-scattering events, leads to the consideration of operators integrating along broken-rays (also called V-lines) (e.g. see \cite{FMS-PhysRev-10, florescu2018, FMS-09, Kats_Krylov-15}). The vertex of each V-line corresponds to a scattering location and, therefore, is inside the support of the image field. Generalized Radon-type transforms, mapping a \textit{scalar function} to a family of its integrals along V-lines with a vertex inside the image support, have attracted substantial interest of researchers during the last two decades (see \cite{amb-book} and the references therein). Numerous interesting results have been obtained about the injectivity, range description, inversion, stability and other properties of the V-line transform, as well as the closely related star transform on scalar fields in $\Rb^2$ (see \cite{amb_2012, amb-lat_2019, Amb_Lat_star, Ambartsoumian_Moon_broken_ray_article, ambartsoumian2016numerical, Florescu-Markel-Schotland, Gouia_Amb_V-line, Kats_Krylov-13, Sherson, walker2019broken, ZSM-star-14} and references therein). 

The V-line and star transforms of \textit{vector fields} in $\Rb^2$ have been analyzed in a pair of recent works \cite{Gaik_Mohammad_Rohit, Gaik_Mohammad_Rohit_2023}. The authors extended the notions of the longitudinal, transverse, and momentum ray transforms from operators integrating vector fields along straight lines to those integrating along broken-ray trajectories. They described the kernels of these operators and derived several exact formulas for recovering a vector field from certain combinations of the aforementioned transforms. In addition to that, the notion of the star transform was extended from scalar fields (functions) to vector fields in $\Rb^2$, and an exact closed form inversion formula was derived for that operator. Building up on the previous results, the current work extends the notions of the V-line and star transforms to the \textit{symmetric 2-tensor fields} in $\Rb^2$. We present explicit characterizations of the kernels of the VLTs and derive exact formulas for reconstruction of tensor fields from various combinations of those transforms. In addition to that, we describe all injective configurations of the star transform and derive an exact, closed-form inversion formula for that operator.

The paper is organized as follows. In Section \ref{sec:def}, we define various V-line transforms on symmetric 2-tensor fields in $\mathbb{R}^2$, introduce the necessary notations, and prove some preliminary results. In Section \ref{sec:Kernels}, we characterize the kernels of the operators of our interest. In Section \ref{sec:Special}, we provide reconstruction results for some special kinds of tensor fields. In Section \ref{sec: reconstruction L, T, and M}, we show the full reconstruction of a symmetric 2-tensor field from its longitudinal, transverse, and mixed V-line transforms. In Section \ref{sec:moments}, we demonstrate the full reconstruction of a symmetric 2-tensor field from the longitudinal, transverse and/or mixed V-line transforms combined with one of their first moment transforms. In Section \ref{sec:Star}, we introduce the star transform  on symmetric 2-tensor fields in the plane and derive a closed-form inversion formula for that operator. Section \ref{sec:add_remarks} contains some additional remarks. We finish the paper with acknowledgements in Section \ref{sec:acknowledge}.

%%%%%%%%%%%%%%%%%%%%%%%%%%%%%%%%%%%%%%%%%%%%%%%%%%%%%%
\section{Definitions, notations, and preliminary results}\label{sec:def}
Throughout this paper, we use bold font letters to denote vectors and 2-tensors in $\mathbb{R}^2$ (e.g. $\vx$, $\textbf{\textit{u}}$, $\textbf{\textit{v}}$, $\vf$, etc.), and regular font letters to denote scalars (e.g. $x_i$, $h$, $f_{ij}$, etc). We denote by $\vx\cdot\vy$ the usual dot product between vectors $\vx$ and $\vy$.

Let $S^2(\Omega)$ denote the space of symmetric $2$-tensor fields defined in a domain $\Omega\subseteq\mathbb{R}^2$, and $C_c^2\left(S^2;\Omega\right)$ be the space of twice continuously differentiable, compactly supported, symmetric $2$-tensor fields. In local coordinates, $\vf \in C_c^2\left(S^2;\Omega\right)$ can be expressed as 
\begin{align}\label{eq:coordinate rep of f}
  \vf(x) = f_{ij}(x) \,dx^{i} dx^{j},  
\end{align}
where $f_{ij}(x)$ are compactly supported $C^2$ functions such that  $f_{ij} =  f_{ji}$ for all indices $i$ and $j$. The Einstein summation convention is assumed throughout the article.

The scalar product in $S^2(\Omega)$ is defined by the formula
\begin{align}\label{eq:inner product in Sm}
\langle \vf, \textbf{\textit{g}}\rangle =  f_{ij}g^{ij} = f_{11}g^{11} + 2 f_{12}g^{12} + f_{22}g^{22}.
\end{align}
\noindent For a scalar function $V(x_1, x_2)$  and a vector field $\vf =(f_1,f_2)$, we use the notations 
\begin{align}\label{eq: definition of div and curl}
\D V = \left(\frac{\partial V}{\partial x_1}, \frac{\partial V}{\partial x_2}\right), \ \  \D^\perp V = \left(-\frac{\partial V}{\partial x_2}, \frac{\partial V}{\partial x_1}\right), \ \    \delta \vf =  \frac{\partial f_1}{\partial x_1}+ \frac{\partial f_2}{\partial x_2},\ \  
\delta^\perp \vf =  \frac{\partial f_2}{\partial x_1}- \frac{\partial f_1}{\partial x_2}.
\end{align}

The operators $\D$ and $\delta$ are the classical gradient and divergence operators, respectively, while  the operators $\D^\perp$  and $\delta^\perp$ are the corresponding ``orthogonal" operators. These operators can be generalized naturally to higher-order tensor fields in the following way\footnote{These generalized differential operators have been defined earlier in the straight line setup in \cite{derevtsov3}.}: 
\begin{align}
(\D \vf)_{ij} &= \frac{1}{2}\left(\frac{\partial f_i}{\partial x_j} +\frac{\partial f_j}{\partial x_i}\right), \quad \vf \mbox{ is a vector field,}\label{eq:symm-der} \\
(\D^\perp \vf)_{ij} &= \frac{1}{2}\left((-1)^j\frac{\partial f_i}{\partial x_{3-j}} + (-1)^i \frac{\partial f_j}{\partial x_{3-i}}\right), \quad \vf \mbox{ is a vector field,}\\
(\delta \vf)_i &= \frac{f_{i1}}{\partial x_1} + \frac{f_{i2}}{\partial x_2} =  \frac{f_{ij}}{\partial x_j}, \quad \vf \mbox{ is a symmetric 2-tensor field,}\\
(\delta^\perp \vf)_i &= - \frac{f_{i1}}{\partial x_2} + \frac{f_{i2}}{\partial x_1} =  (-1)^j\frac{f_{ij}}{\partial x_{3-j}}, \quad \vf \mbox{ is a symmetric 2-tensor field.} \label{eq:def of deltaprep}
\end{align}
\begin{rem}
    The generalized operators $\D$ and $\delta$ defined above are known as symmetrized derivative/inner differentiation and divergence on tensor fields, respectively (e.g., see \cite[p. 25]{Sharafutdinov_Book}). 
\end{rem}
Let  $\vu$ and $\vv$ be two linearly independent unit vectors in $\mathbb{R}^2$. For $\vx \in \mathbb{R}^2$, the rays emanating from $\vx$ in directions $\vu$ and $\vv$ are denoted by  $L_{\vu}(\vx)$ and $L_{\vv}(\vx)$ respectively, i.e.
$$ L_{\vu}(\textbf{\textit{x}}) = \left\{\textbf{\textit{x}} +t \textbf{\textit{u}}: 0 \leq t < \infty\right\} \quad \mbox{ and } \quad  L_{\vv}(\textbf{\textit{x}}) = \left\{\textbf{\textit{x}} +t \textbf{\textit{v}}: 0 \leq t < \infty\right\}.$$
A V-line with the vertex $\textbf{\textit{x}}$ is the union of rays $L_{\vu}(\textbf{\textit{x}})$ and $L_{\vv}(\textbf{\textit{x}})$. Through the rest of the article we will assume that  $\textbf{\textit{u}}$ and $\textbf{\textit{v}}$ are fixed, i.e. all V-lines have the same ray directions and can be parametrized simply by the coordinates $\textbf{\textit{x}}$ of the vertex (see Figure \ref{fig1a}). 
\begin{figure}[H]
\begin{center}
\begin{subfigure}{.45\textwidth}
\centering
 \includegraphics[height=3.8cm]{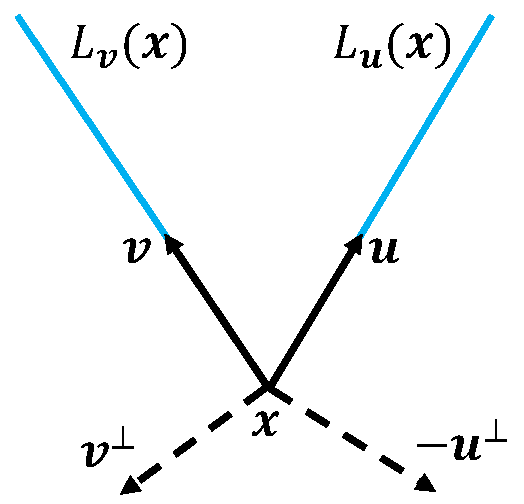} 
 \caption{A V-line with vertex at $\vx$, ray directions $\vu$, $\vv$ and outward normals $-\vu^\perp$, $\vv^\perp$.}
  \label{fig1a}
\end{subfigure} \qquad 
\begin{subfigure}{.45\textwidth}
\centering
\includegraphics[height=3.8cm]{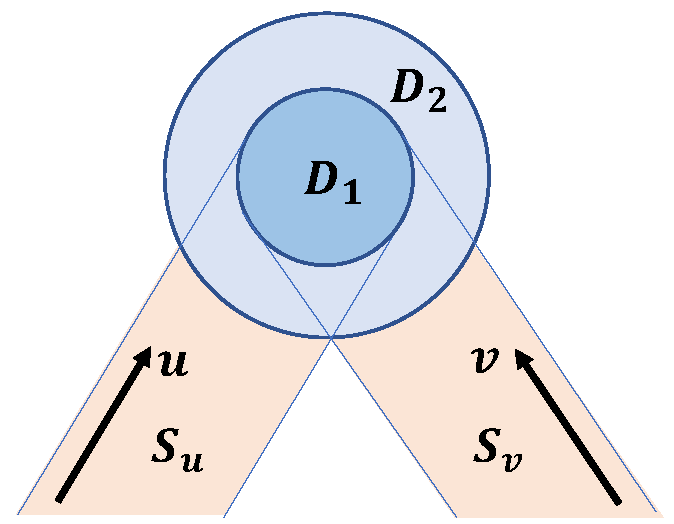}
\caption{A sketch of the compact support of  $\vf$ and the unbounded support of $\Lc\vf$, $\Tc\vf$, $\Mc\vf$.} \label{fig1b}
\end{subfigure}
\end{center}
\vspace{-5mm}
\caption{From \cite{Gaik_Mohammad_Rohit}.} 
\label{fig1} 
\end{figure}

\begin{defn}
 The \textbf{divergent beam transform} $\mathcal{X}_{\vu}$ of function $h$ at $\vx\in \mathbb{R}^2$ in the direction $\vu\in\Sb^1$ is defined as:
 \begin{equation}\label{def:DivBeam}
   \mathcal{X}_{\vu}h(\vx) =  \int_{0}^{\infty} h(\vx+t \vu)\,dt.
 \end{equation}
 \end{defn}
\begin{defn}
Let $k\ge0$ be in an integer. The \textit{\textbf{$k$-th moment  divergent beam transform}}\footnote{It is easy to see that $\Xc^0_{\vu} h=\Xc_{\vu} h$. While in this article we use primarily the latter notation, the former becomes handy in some recurrence relations discussed in Section \ref{sec:moments}.} of a function $h$ in the direction $\vu\in\Sb^1$ is defined as follows 
 \begin{equation}\label{def:moment DivBeam}
  \Xc^k_{\vu} h =  \int_0^\infty h(\vx + t \vu)\, t^k\,  dt.
 \end{equation}
\end{defn}
The directional derivative of a function in the direction $\vu$ is denoted by $D_{\vu}$, i.e. 
\begin{equation}
D_{\vu} h = \vu \cdot \nabla h.
\end{equation}
For a compactly supported function $h$, the operators  $\mathcal{X}_{\vu}$ and $D_{\vu}$ commute and satisfy the following relation:
\begin{equation}\label{eq:commut}
    \mathcal{X}_{\vu}(D_{\vu}h)(\vx)=D_{\vu} (\mathcal{X}_{\vu}h)(\vx)= -h(\vx).
\end{equation}
\begin{defn}
For  two vectors $\vu=(u_1, u_2)$ and $\vv = (v_1, v_2)$, the tensor product $\vu \otimes \vv$ is a rank-2 tensor with its $ij$-th component defined by
\begin{align}\label{eq:tensor product}
(\vu \otimes \vv)_{ij}   =  u_iv_j.
\end{align}
The symmetrized tensor product $\vu \odot \vv$
is then defined as
\begin{align}\label{eq:symmetrized tensor product}
\vu \odot \vv  =  \frac{1}{2}\left(\vu \otimes \vv + \vv \otimes \vu \right).
\end{align}
\end{defn}
\noindent We use the notation $\vu^2$ for the tensor (symmetrized) product of a vector $\vu$ with itself, that is, 
$$ \vu^2 = \vu \odot \vu = \vu \otimes \vu.$$
%%%%%%%%%%%%%%%%%%%%%%%%%%%%%%%%%%%%%%%%%%%%%%%%

In this paper, our main goal is to recover a symmetric 2-tensor field from its V-line transforms\footnote{The name V-line is due to the resemblance of the paths of integration to the letter V.}, which we define below. 

\begin{defn} \label{def:definition of longitudinal V line transform} Let $\vf\in C_c^2\left(S^2;\mathbb{R}^2\right)$.
The \textbf{longitudinal V-line transform} of $\textbf{\textit{f}}$ is defined as
\begin{align}\label{eq:def V-line transform}
\mathcal{L}_{\vu, \vv}\, \vf\  = \mathcal{X}_{\vu}  \left( \left\langle \vf , \vu^2 \right\rangle \right)  + \mathcal{X}_{\vv}  \left( \left\langle \vf , \vv^2 \right\rangle \right).
\end{align}
\end{defn}
To define the second integral transform of interest, we need to make a choice for the normal unit vectors corresponding to each branch of the V-line. We define the vector $\perp$ operation by $(x_1,x_2)^\perp = (-x_2,x_1) $.

%--------------------------------------------------

\begin{defn}\label{def:definition of transverse V line Doppler transform}
Let $\vf\in C_c^2\left(S^2;\mathbb{R}^2\right)$. 
The \textbf{transverse V-line transform} of $\textbf{\textit{f}}$ is defined as
\begin{align}\label{eq:def transverse V-line transform}
\mathcal{T}_{\vu, \vv}\, \vf\  = \mathcal{X}_{\vu}  \left( \left\langle \vf , (\vu^\perp)^2 \right\rangle \right)  + \mathcal{X}_{\vv}  \left( \left\langle \vf , (\vv^\perp)^2 \right\rangle \right).
\end{align}
\end{defn}

\begin{defn}\label{def:definition of mixed V line Doppler transform}
Let $\vf\in C_c^2\left(S^2;\mathbb{R}^2\right)$.
The \textbf{mixed V-line transform} of $\textbf{\textit{f}}$ is defined as
\begin{align}\label{eq:def mixed V-line transform}
\mathcal{M}_{\vu, \vv}\, \vf\  = \mathcal{X}_{\vu}  \left( \left\langle \vf , \vu \odot \vu^\perp \right\rangle \right)  + \mathcal{X}_{\vv}  \left( \left\langle \vf , \vv \odot \vv^\perp \right\rangle \right).
\end{align}
\end{defn}

%------------------------------------------------------
%-----------------------------------------------
\begin{defn} \label{def:definition of k-th V line moment  transform} 
Let $\vf\in C_c^2\left(S^2;\mathbb{R}^2\right)$.
The \textbf{$k$-th moment longitudinal V-line transform} of $\textbf{\textit{f}}$ is defined as
\begin{align}\label{eq:first moment longitudinal V-line transform}
\mathcal{L}^k_{\vu, \vv}\, \vf\  = \mathcal{X}^k_{\vu}  \left( \left\langle \vf , \vu^2 \right\rangle \right)  + \mathcal{X}^k_{\vv}  \left( \left\langle \vf , \vv^2 \right\rangle \right).
\end{align}
\end{defn}

\begin{defn} \label{def:definition of kth V line transverse moment  transform} 
Let $\vf\in C_c^2\left(S^2;\mathbb{R}^2\right)$.
The \textbf{$k$-th moment transverse V-line transform} of $\textbf{\textit{f}}$ is defined as
\begin{align}\label{eq:k-th moment transverse V line transform}
\mathcal{T}^k_{\vu, \vv}\, \vf\  = \mathcal{X}^k_{\vu}  \left( \left\langle \vf , (\vu^\perp)^2 \right\rangle \right)  + \mathcal{X}^k_{\vv}  \left( \left\langle \vf , (\vv^\perp)^2 \right\rangle \right).
\end{align}
\end{defn}

\begin{defn} \label{def:definition of kth V line mixed moment transform} 
Let $\vf\in C_c^2\left(S^2;\mathbb{R}^2\right)$.
The \textbf{$k$-th moment mixed V-line transform} of $\textbf{\textit{f}}$ is defined as
\begin{align}\label{eq:k-th moment mixed V line transform}
\mathcal{M}^k_{\vu, \vv}\, \vf\  = \mathcal{X}^k_{\vu}  \left( \left\langle \vf , \vu \odot \vu^\perp \right\rangle \right)  + \mathcal{X}^k_{\vv}  \left( \left\langle \vf ,\vv \odot \vv^\perp \right\rangle \right).
\end{align}
\end{defn}

%-----------------------------------------------
%-----------------------------------------------

\begin{rem}
Throughout the paper we assume that the linearly independent unit vectors $\vu$ and $\vv$ are fixed. Hence, to simplify the notations, we will drop the indices $\vu,\vv$ and refer to $\mathcal{T}_{\textbf{\textit{u}}, \textbf{\textit{v}}}$, $\mathcal{L}_{\textbf{\textit{u}},  \textbf{\textit{v}}}$, $\mathcal{M}_{\vu, \vv}$, 
$\mathcal{T}^k_{\textbf{\textit{u}}, \textbf{\textit{v}}}$, 
$\mathcal{L}^k_{\textbf{\textit{u}},  \textbf{\textit{v}}}$, and $\mathcal{M}^k_{\vu, \vv}$ 
simply as $\mathcal{T}$, $\mathcal{L}$, $\mathcal{M}$, $\mathcal{T}^k$, $\mathcal{L}^k$, and $\mathcal{M}^k$.
\end{rem}
%-----------------------------------------------

Let us assume that  $\operatorname{supp}\vf\subseteq D_1$, where $D_1$ is an open disc of radius $r_1$ centered at the origin. Then $\mathcal{T}\vf$, $\mathcal{L}\vf$, $\mathcal{M}\vf$, and their moments are supported inside an unbounded domain $D_2\cup S_{\vu} \cup S_{\vv}$, where $D_2$ is a disc of some finite radius $r_2>r_1$ centered at the origin, while $S_{\vu}$ and $S_{\vv}$ are semi-infinite strips (outside of $D_2$) in the direction of $\vu$ and $\vv$, respectively (see Figure \ref{fig1b}). It is easy to notice that $\Lc\vf, \Tc\vf$, and $\Mc \vf$ are constant along the rays in the directions $\vu$ and $\vv$ inside the corresponding strips $S_{\vu}$ and $S_{\vv}$. In other words, their restrictions to $\overline{D_2}$ completely define them in $\mathbb{R}^2$. If one also knows the corresponding moments in $\overline{D_2}$, then the values of the moments are uniquely determined everywhere else. Indeed, let $\vy$ be inside the strip $S_{\vu}$, and $\vx$ be a point on the boundary between $D_2$ and $S_{\vu}$ such that $\vy+a\vu=\vx$ for some constant $a\textgreater 0$. It is easy to see that 
\begin{equation}
    \Lc^1\vf(\vy)=\Lc^1\vf(\vx)+ a \Lc\vf(\vx).
\end{equation}
Similar result follows for transverse and mixed moment V-line transforms.

\begin{rem}
In this article we assume that  the tensor field $\vf$ is supported in $D_1$ and the transforms $\Lc\vf\,(\vx)$, $\Tc\vf\,(\vx)$, $\Mc\vf\,{(\vx)}$, and their moments are known for all $\vx\in \overline{D_2}$.
\end{rem}
\begin{rem}\label{rem:choice of u and v}
To simplify various calculations, we assume $\vu = (u_1, u_2)$ and $\vv =  (-u_1, u_2)$, that is, the V-lines are symmetric with respect to the $y$-axis. This choice does not change the analysis of the general case, since the data obtained in one setup of $\vu$ and $\vv$ can be transformed into the data obtained for the other setup and vice versa.
\end{rem}
\begin{rem}
    If $\vu=(0,1)=\vv$, the V-line degenerates into a pair of coinciding rays. If $\vu=(1,0)=-\vv$, the V-line becomes a straight line. Since we are interested in proper V-lines, we avoid both of these extreme cases. Therefore, we assume $-1<u_1^2-u_2^2<1$.
\end{rem}
%%%%%%%%%%%%%%%%%%%%%%%%%%%%%%%%%%%%%%%%%%%%%%%%%%%%%%%%%
With this choice of the parameters of the V-lines, we have the following identities (which can be verified by direct calculations):
\begin{align*}
\langle \vf, \vu^2 \rangle &=   u_1^2f_{11} + 2 u_1 u_2 f_{12}  + u_2^2f_{22},\\
\langle \vf, \vu_\perp^2 \rangle &=   u_2^2f_{11} - 2 u_1 u_2 f_{12}  +u_1^2f_{22},\\
\langle \vf, \vu\odot \vu_\perp \rangle &=  -u_1 u_2 f_{11} + (u_1^2 -u_2^2)f_{12} +u_1 u_2 f_{22}, \\
\langle \vf, \vv^2 \rangle &=  u_1^2f_{11} - 2 u_1 u_2 f_{12}  + u_2^2f_{22},\\
\langle \vf, \vv_\perp^2 \rangle &=   u_2^2f_{11} + 2 u_1 u_2 f_{12}  + u_1^2f_{22},\\
\langle \vf, \vv\odot \vv_\perp \rangle &=  u_1 u_2 f_{11} + (u_1^2 -u_2^2)f_{12} -u_1 u_2 f_{22}.
\end{align*}
Using these identities, we can rewrite our integral operators in the following simplified forms: 
\begin{align}
\Lc \vf (\vx) &= u_1^2 \Vc(f_{11})(\vx) + 2u_1u_2 \Vc^{-}(f_{12})(\vx) + u_2^2 \Vc(f_{22})(\vx),\label{eq: Lf}\\
\Tc \vf (\vx) &= u_2^2 \Vc(f_{11})(\vx) - 2u_1u_2 \Vc^{-}(f_{12})(\vx) + u_1^2 \Vc(f_{22})(\vx),\label{eq:Tf}\\
\Mc \vf (\vx) &= -u_1u_2 \Vc^{-}(f_{11})(\vx) + (u_1^2 -u^2_2) \Vc(f_{12})(\vx) + u_1u_2 \Vc^{-}(f_{22})(\vx),\label{eq:Mf}
\end{align}
where $\Vc \varphi$ and $\Vc^{-} \varphi$ denote  the V-line transform and the \textit{signed} V-line transform of scalar function $\varphi$. Namely, 
\begin{align}\label{def:V-line transform}
    \Vc \varphi(\vx) =  \Xc_{\vu} \varphi (\vx) + \Xc_{\vv}\varphi(\vx),
\end{align}
and
\begin{align}\label{def:Signed V-line transform}
    \Vc^{-} \varphi(\vx) =  \Xc_{\vu} \varphi (\vx) - \Xc_{\vv}\varphi(\vx).
\end{align}
\begin{theorem}\label{th:trace recovery}
Let $\vf\in C_c^2\left(S^2;D_1\right)$.
Then the knowledge of $\Lc \vf$ and $\Tc \vf$ determines the trace of $\vf$ as follows:
\begin{align}\label{eq:reconstruction of trace of f}
\left(f_{11} + f_{22}\right)(\vx) = \frac{1}{2 u_2} D_\vu D_\vv  \int_{0}^{\infty}\left(\Lc \vf + \Tc \vf \right)(\vx+ t \vev_2)\,dt,
\end{align}
where $\vev_2=(0,1)$.
\end{theorem}
\begin{proof}
By adding equations \eqref{eq: Lf} and \eqref{eq:Tf}, we get the following relation:
$$ \Lc \vf (\vx) + \Tc \vf (\vx) = \Vc \left(f_{11} + f_{22}\right)(\vx). $$
Using the inversion of the V-line transform given in \cite[Corollary 2]{amb-lat_2019}, we obtain the required inversion formula \eqref{eq:reconstruction of trace of f}.
\end{proof}

One can easily observe that, when the branches of the V-lines are orthogonal to each other (or, equivalently, when $u_1=u_2$), the symmetric tensor products $\vu^2$, $\vu_\perp^2$, $\vu\odot \vu_\perp$, $\vv^2$, $\vv_\perp^2$, and $\vv\odot \vv_\perp$ become degenerate. As it will be shown throughout the paper, this special case exhibits a qualitatively different behaviour in the problems of recovering $\vf$ from its integral transforms. For example, the results of Theorem \ref{th:trace recovery}, in that particular setup, can be augmented as follows.

%---------------------------------------

\begin{theorem}\label{th-special_case}
Let $\vf\in C_c^2\left(S^2;D_1\right)$ and $u_1=u_2$.
The knowledge of $\Lc \vf$ and $\Tc \vf$ determines the off-diagonal component $f_{12}$ of $\vf$ as follows:
\begin{equation}\label{eq:f12-special}
    f_{12}(\vx) = \frac{1}{4 u_1} D_\vu D_\vv  \int_{0}^{\infty}\left(\Lc \vf - \Tc \vf \right)(\vx- t \vev_1)\,dt,
\end{equation}
where $\vev_1=(1,0)$, while $\Mc\vf$ provides the difference of the diagonal components of $\vf$ through:
\begin{equation}\label{eq:f22-f11-special}
    \left(f_{22}-f_{11}\right)(\vx) = \frac{1}{u_1} D_\vu D_\vv  \int_{0}^{\infty}\Mc \vf\, (\vx- t \vev_1)\,dt.
\end{equation}
\end{theorem}

\begin{proof}
By subtracting equation \eqref{eq:Tf} from equation \eqref{eq: Lf}, we get the following relation:
$$ 
\Lc \vf (\vx) - \Tc \vf (\vx) = 4 u_1^2\,\Vc ^{-}f_{12}(\vx). 
$$
Under the assumption $u_1=u_2$, equation (\ref{eq:Mf}) simplifies into the expression:
$$
\Mc \vf (\vx) = u_1^2\; \Vc^{-}(f_{22}-f_{11})(\vx).
$$
Using the inversion of the signed V-line transform given in \cite[Theorem 7]{amb-lat_2019}, we obtain both of the required inversion formulas \eqref{eq:f12-special} and \eqref{eq:f22-f11-special}.
\end{proof}

\begin{cor}\label{cor: u1 =u2 case}
If $u_1=u_2$, then $\vf\in C_c^2\left(S^2;D_1\right)$ can be recovered from $\Lc \vf$, $\Tc \vf$, and $\Mc \vf$.    
\end{cor}

%%%%%%%%%%%%%%%%%%%%%%%%%%%%%%%%%%%%%%%%%%%%%%%%%

\section{Kernel descriptions of transforms $\Lc$, $\Tc$, and $\Mc$\label{sec:Kernels}}

This section is devoted to the study of the kernels of the VLTs defined on $C_c^2\left(S^2;\mathbb{R}^2\right)$. Each of these integral transforms has a non-trivial null space, which we characterize below. 
Our findings here are quite different from the known results for the usual straight-line transforms. For instance, the potential tensor fields are in the kernel of the longitudinal ray transform (along straight lines) \cite{Sharafutdinov_Book}, which is not the case in the V-line setup, as one can easily check. 
\begin{theorem}\label{th:kernel description}
Let $\vf\in C_c^2\left(S^2;D_1\right)$.
Then
\begin{enumerate}
    \item[(a)] \label{th:kernel of L} $\vf$ is in the kernel of $\Lc$ if and only if there exists a scalar function $\varphi$ such that 
\begin{align}\label{eq: kernel of L}
\D \varphi =  \left(u_1^2 f_{11} + u_2^2 f_{22}, 2 u_1^2 f_{12}\right).
\end{align}
\item[(b)] $\vf$ is in the kernel of $\Tc$ if and only if there exists a scalar function $\varphi$ such that 
\begin{align}\label{eq: kernel of T}
\D \varphi =  \left(u_2^2 f_{11} + u_1^2 f_{22}, -2 u_1^2 f_{12}\right).
\end{align}
\item[(c)] $\vf$ is in the kernel of $\Mc$ if and only if there exists a scalar function $\varphi$ such that 
\begin{align}\label{eq: kernel of M}
\D \varphi =  \left((u_1^2 - u_2^2) f_{12} , - u_1^2 (f_{11} - f_{22})\right).
\end{align}
\end{enumerate}
\end{theorem}
\begin{proof}
\noindent \textbf{\textit{Part (a).}} As a first step, we observe that $\Lc \vf (\vx) =0 $ if and only if  $D_\vu D_\vv  \Lc \vf(\vx) =0$. 
This follows from formula (\ref{eq:commut}) after applying $\mathcal{X}_{\vv}\mathcal{X}_{\vu}$ to the second relation.
Now, consider
\begin{align*}
D_\vu D_\vv  \Lc \vf = 0 & \Longleftrightarrow   4 u_1^2 u_2 \partial_{x_{1}}  f_{12} -2 u_1^2 u_2 \partial_{x_{2}} f_{11} - 2u_2^3  \partial_{x_{2}}f_{22} =0 \\
& \Longleftrightarrow \operatorname{curl} \left( u_1^2  f_{11} + u_2^2 f_{22},   2 u_1^2  f_{12}  \right) =0.
 \end{align*}
It is known that $\operatorname{curl}(\vg)=\partial_{x_{1}}g_2-\partial_{x_{2}}g_1$ of a 2D vector field $\vg$ vanishes if and only if $ \vg =  \D \varphi$ for some scalar function $\varphi$. Therefore, there exists a scalar function $\varphi$ such that
$$\D \varphi = \left(  u_1^2  f_{11} + u_2^2 f_{22}, 2 u_1^2  f_{12}\right).$$

\noindent \textbf{\textit{Part (b).}} Proceeding as in the previous part, observe that $\Tc \vf (\vx) =0$ iff  $-D_\vu D_\vv  \Tc \vf(\vx) =0$.
\begin{align}
-D_\vu D_\vv  \Tc \vf = 0  & \Longleftrightarrow  4 u_1^2 u_2 \partial_{x_{1}}  f_{12} +2 u_2^3 \partial_{x_{2}} f_{11} + 2u_1^2u_2  \partial_{x_{2}}f_{22} =0\nonumber \\
& \Longleftrightarrow \operatorname{div}\left(  2 u_1^2  f_{12} , u_2^2  f_{11} + u_1^2 f_{22} \right) = 0 \tblue{.}
 \end{align}
 For every two-dimensional  solenoidal vector field  $\vg$  on the unit disk, there exists a  scalar function $\varphi$ such that $ \vg= \D^\perp \varphi.$ Hence, $\vf$ is in the kernel of $\Tc$ if and only if there exists a $\varphi$ such that 
$$ \D^\perp \varphi = \left( 2 u_1^2  f_{12},  u_2^2  f_{11} + u_1^2 f_{22} \right) \Longleftrightarrow \D \varphi = \left( u_2^2  f_{11} + u_1^2 f_{22},  - 2 u_1^2  f_{12} \right).$$

\noindent \textbf{\textit{Part (c).}}
\noindent  Again, note that $\Mc \vf (\vx) = 0 $ if and only if  $-D_\vu D_\vv  \Mc \vf(\vx) =0$. Then  
\begin{align*}
-D_\vu D_\vv  \Mc \vf = 0 & \Longleftrightarrow 2 u_1^2 u_2 \partial_{x_1}  f_{11} - 2 u_1^2 u_2 \partial_{x_1} f_{22} + 2u_2 (u_1^2 -u_2^2)\partial_{x_2} f_{12} =0\\
& \Longleftrightarrow \operatorname{div}\left(  u_1^2 \left( f_{11} -f_{22}\right), \left( u_1^2 - u_2^2\right) f_{12} \right) =0. 
 \end{align*}
Therefore, $\vf$ is in the kernel of $\Mc$ if and only if there exists a scalar function $\varphi$ 
such that 
$$  
\D^\perp \varphi = \left( u_1^2 \left( f_{11} -f_{22}\right), \left( u_1^2 - u_2^2\right) f_{12} \right) \Longleftrightarrow \D \varphi = \left(\left( u_1^2 - u_2^2\right) f_{12},  - u_1^2 \left( f_{11} -f_{22}\right) \right),
$$
which finishes the proof.
\end{proof}
%%%%%%%%%%%%%%%%%%%%%%%%%%%%%%%%%%%%%%%%%%%%%%%%%%%%%%%%
\section{Reconstruction results for special kinds of tensor fields}\label{sec:Special}
In this section we focus on symmetric 2-tensor fields of the form $\D \vg$ or $\D^\perp \vg $, where $\vg$ is a vector field in $\mathbb{R}^2$. We show that such tensor fields can be recovered from the knowledge of specific combinations of their VLTs. Furthermore, if the unknown tensor field is of the form $\D^2 \varphi$, $\D\D^\perp \varphi$ or $(\D^\perp)^2\varphi$, where $\varphi$ is a scalar function, then the field can be uniquely reconstructed just from one integral transform, depending on the form of $\vf$.  
\begin{theorem}\label{th:special tensor recovery 1} Let $\vg = (g_1, g_2)$ be a vector field with components $g_i(\vx) \in C_c^2(D_1)$, for $i =  1, 2$.
\begin{itemize}
\item[(a)] If $\vf$ is a symmetric 2-tensor field of the form $\vf = \D \vg$, then it can be recovered explicitly in terms of $\Lc \vf$ and $\Mc \vf$.
\item[(b)] If $\vf$ is a symmetric 2-tensor field of the form $\vf = \D^\perp \vg$, then it can be recovered explicitly in terms of $\Tc \vf$ and $\Mc \vf$. 
\end{itemize}
\end{theorem}
\begin{proof} \textbf{\textit{Part (a).}}
From the definition of the symmetrized derivative \eqref{eq:symm-der}, we have 
     $$  f_{ij} = (\D \vg)_{ij} = \frac{1}{2}\left(\frac{\partial g_i}{\partial x_j} +\frac{\partial g_j}{\partial x_i}\right).$$ 
 A direct computation yields 
   \begin{align}\label{eq: L on df}
      \int_{0}^{\infty} f_{ij}(\vx+t \vu) u^{i} u^{j}\,dt 
 = \frac{1}{2} \int_{0}^{\infty} \left(\frac{\partial g_i}{\partial x_j} +\frac{\partial g_j}{\partial x_i}\right) (\vx+t \vu) u^{i} u^{j}\,dt 
   = -\left\langle \vg(\vx) , \vu \right\rangle ,
    \end{align}
   \begin{align} \label{eq: M on df} 
 \int_{0}^{\infty} f_{ij}(\vx+t \vu) u^{i} (u^\perp)^{j}\,dt 
 &= \frac{1}{2} \int_{0}^{\infty} \left(\frac{\partial g_i}{\partial x_j} +\frac{\partial g_j}{\partial x_i}\right) (\vx+t \vu) u^{i} (u^\perp)^{j}\,dt \nonumber \\
    &= -\left\langle \vg(\vx) , \vu^\perp \right\rangle - \frac{1}{2} \Xc_\vu (\delta^\perp \vg)(\vx).
    \end{align}
Combining formulas \eqref{eq: L on df} and \eqref{eq: M on df} we get the following expressions for $\Lc \vf (\vx)$, and $ \Mc \vf (\vx)$:
\begin{align}\label{eq: L and M on df}
    \Lc \vf (\vx) = - \left\langle \vg(\vx) , \vu + \vv \right\rangle \quad \mbox{ and } \quad 
    \Mc \vf (\vx) = - \left\langle \vg(\vx) , \left(\vu + \vv\right)^\perp \right\rangle - \frac{1}{2} \Vc \left( \delta^\perp \vg\right)(\vx).
\end{align}
With our choice of $\vu = (u_1,u_2)$ and $\vv = (-u_1,u_2)$ (see Remark \ref{rem:choice of u and v}) , we have $\vu + \vv = (0,2u_2)$ and $\left(\vu + \vv \right)^\perp = (-2u_2, 0)$. Using these relations, we get
$$ g_2(\vx) =  - \frac{1}{2u_2}\Lc \vf(\vx)\quad  \mbox{ and }\quad \Mc \vf (\vx) =  2 u_2 g_1(\vx) - \frac{1}{2}\Vc \left(\frac{\partial g_2}{\partial x_1} - \frac{\partial g_1}{\partial x_2}\right) (\vx).  
$$
Thus $g_2$ is known from $\Lc \vf$, and the aim now is to recover $g_1$ using $g_2$ and $\Mc \vf$. 
Consider
\begin{align*}
\Vc \left(\frac{\partial g_1}{\partial x_2}\right) (\vx)  + 4 u_2 g_1(\vx) &=   \underbrace{2\Mc \vf (\vx) + \Vc \left(\frac{\partial g_2}{\partial x_1}\right) (\vx)}_{h(\vx) \ \mbox{known data}}\\
 \Rightarrow \quad \ \ 
D_\vu D_\vv(\Xc_\vu + \Xc_\vv)  \left(\frac{\partial g_1}{\partial x_2}\right) (\vx)  + 4 u_2  D_\vu D_\vv g_1(\vx) &=   D_\vu D_\vv h (\vx)\\
  \Rightarrow \qquad \quad  \  -(D_\vv + D_\vu) \left(\frac{\partial g_1}{\partial x_2}\right) (\vx)  + 4 u_2  D_\vu D_\vv g_1(\vx) &=   D_\vu D_\vv h (\vx)\\
\Rightarrow \quad \ \  \  -2 u_2 \frac{\partial^2 g_1}{\partial x_2^2}(\vx)  + 4 u_2  \left(u_2^2 \frac{\partial^2}{\partial x_2^2}  -  u_1^2 \frac{\partial^2}{\partial x_1^2}\right) g_1(\vx) &=   D_\vu D_\vv h (\vx)\\
 \Rightarrow \qquad \quad  \qquad\qquad \  \  2 u_1^2 \frac{\partial^2 g_1}{\partial x_1^2}(\vx)  + (1 - 2u_2^2)  \frac{\partial^2 g_1}{\partial x_2^2}(\vx)  &=   - \frac{1}{2u_2}D_\vu D_\vv h (\vx)\\
  \Rightarrow \qquad \qquad \quad  \qquad \  \  2 u_1^2 \frac{\partial^2 g_1}{\partial x_1^2}(\vx)  + (u_1^2 - u_2^2)  \frac{\partial^2 g_1}{\partial x_2^2}(\vx)  &=   - \frac{1}{2u_2}D_\vu D_\vv h (\vx). 
\end{align*}
Therefore $g_1$ satisfies a second order partial differential equation with additional homogeneous boundary (or initial) conditions, which can be uniquely solved to get $g_1$. The homogeneous boundary (or initial) conditions arise naturally from the compactness of the support of $\vg$ required in the hypothesis of the theorem. More  explicitly, we have the following three cases:
\begin{itemize}
    \item When $u_1^2 =  u_2^2$, the above partial differential equation becomes:
    $$ 2 u_1^2 \frac{\partial^2 g_1}{\partial x_1^2}(\vx)   =   - \frac{1}{2u_2}D_\vu D_\vv h (\vx), $$
    which gives $g_1$ by integrating twice along $\vev_1$-direction. 
\item When $u_1^2 >  u_2^2$, we have an elliptic  PDE
 \begin{align*}
        2u_1^2\frac{\partial^2 g_1}{\partial x_1^2} + \left(u_1^2 - u_2^2\right) \frac{\partial^2 g_1}{\partial x_2^2} = - \frac{1}{2u_2}D_\vu D_\vv h (\vx)
 \end{align*}
 with homogeneous boundary conditions, which has a unique solution.
 \item When $u_1^2 <  u_2^2$, we have a hyperbolic PDE  \begin{align*}
        2u_1^2\frac{\partial^2 g_1}{\partial x_1^2} - \left(u_2^2 - u_1^2\right) \frac{\partial^2 g_1}{\partial x_2^2} &= - \frac{1}{2u_2}D_\vu D_\vv h (\vx), 
 \end{align*}
 which can be solved by choosing appropriate initial conditions on $g_1$ (for instance, we can take $g_1(a, y) = 0$ and $\displaystyle \frac{\partial g_1}{\partial x_1} (a, y) = 0$, for $y\in \mathbb{R}$ and any fixed $a \in \mathbb{R}\setminus [-1, 1]$).   
\end{itemize}
\vspace{1mm}

\noindent\textbf{\textit{Part (b).}} Recall 
       $$ f_{ij} = (\D^ \perp \vg)_{ij} =  \frac{1}{2}\left((-1)^j\frac{\partial g_i}{\partial x_{3-j}} + (-1)^i \frac{\partial g_j}{\partial x_{3-i}}\right).$$ 
Again, by a direct computation, we have 
      \begin{align} \label{eq: T on  dperpf}
       \int_{0}^{\infty} f_{ij}(\vx+t \vu) (u^ \perp)^{i} (u^\perp)^{j}\,dt 
 &= \frac{1}{2} \int_{0}^{\infty} \left((-1)^j\frac{\partial g_i}{\partial x_{3-j}} + (-1)^i \frac{\partial g_j}{\partial x_{3-i}}\right) (\vx+t \vu) (u^\perp)^{i} (u^ \perp)^{j}\,dt \nonumber\\
   & = -\left\langle \vg(\vx) , \vu ^ \perp \right\rangle ,
    \end{align}
    \begin{align} \label{eq: M on  dperpf}
     \int_{0}^{\infty} f_{ij}(\vx+t \vu) u^{i} (u^\perp)^{j}\,dt 
 &= \frac{1}{2} \int_{0}^{\infty} \left((-1)^j\frac{\partial g_i}{\partial x_{3-j}} + (-1)^i \frac{\partial g_j}{\partial x_{3-i}}\right) (\vx+t \vu) u^{i} (u^ \perp)^{j}\,dt \nonumber\\
    &= -\left\langle \vg(\vx) , \vu  \right\rangle - \frac{1}{2}\Xc_\vu \left(\delta \vg\right) (\vx).  \end{align}
As in the previous part, we combine formulas \eqref{eq: T on  dperpf}  and \eqref{eq: M on  dperpf} to get the following expressions:
\begin{align}\label{eq: T and M on dperpf}
    \Tc \vf (\vx) = - \left\langle \vg(\vx), \left(\vu + \vv\right)^\perp \right\rangle \quad \mbox{ and }\quad 
    \Mc \vf (\vx) = - \left\langle \vg(\vx) , \vu + \vv \right\rangle - \frac{1}{2} \Vc \left( \delta \vg\right)(\vx).  
\end{align}
Finally, we use the identities $\left(\vu + \vv\right) =  \left(0, 2 u_2\right)$ and $\left(\vu + \vv\right)^\perp =  \left(-2u_2, 0\right)$ together with the above relations to get 
{$$ g_1(\vx) = \frac{1}{2u_2}\Tc \vf(\vx) 
\quad \mbox{ and }\quad  \Mc \vf (\vx) = - 2 u_2 g_2(\vx) - \frac{1}{2}\Vc \left(\frac{\partial g_1}{\partial x_1} + \frac{\partial g_2}{\partial x_2}\right) (\vx). $$
The first relation reveals $g_1(\vx)$ in terms of $\Tc \vf(\vx)$. Following the ideas used in Part (a), we recover $g_2$ by solving the following  second order partial differential equation (equipped with appropriate boundary or initial conditions):
\begin{align*}
    2 u_1^2 \frac{\partial^2 g_2}{\partial x_1^2}(\vx)  + (u_1^2 - u_2^2)  \frac{\partial^2 g_2}{\partial x_2^2}(\vx)  =   - \frac{1}{2u_2}D_\vu D_\vv \Tilde{h} (\vx),
\end{align*}
where $\displaystyle \Tilde{h}(\vx) = -2\Mc \vf(\vx) - \Vc\left(\frac{\partial g_1}{\partial x_1}\right)(\vx)$.}
\end{proof}

%%%%%%%%%%%%%%%%%%%%%%%%%%%%%%%%%%%%%%%%%%%%%%%%%%%%%%%%%%%%%%%%%%%%%%
\begin{theorem} Let $\varphi$ be a twice differentiable compactly supported function, that is, $\varphi \in C^2_c(D_1)$. 
\begin{itemize}
\item[(a)] If  $\vf$ is a symmetric 2-tensor field of the form $\vf = \D^2 \varphi$, then it can be reconstructed explicitly in terms of $\Lc \vf$ or $\Mc \vf$,
using the following formulas:  
$$ \varphi(\vx)= \frac{1}{2u_2}  \int_0^\infty \Lc \vf(\vx + s\vev_2)ds=- \frac{1}{2u_2}  \int_0^\infty \Mc \vf(\vx + s\vev_1)ds .
$$
\item[(b)]  If  $\vf$ is a symmetric 2-tensor field of the form $\vf = \left(\D^\perp\right)^2 \varphi$, then it can be reconstructed explicitly in terms of $\Tc \vf$ or $\Mc \vf$,
using the following formulas:  
$$ \varphi(\vx)= \frac{1}{2u_2}  \int_0^\infty \Tc \vf(\vx + s\vev_2)ds =  \frac{1}{2u_2}  \int_0^\infty \Mc \vf(\vx + s\vev_1)ds.  
$$
\item[(c)]  If  $\vf$ is a symmetric 2-tensor field of the form $\vf = \D\D^\perp \varphi$, then it can be reconstructed explicitly in terms of $\Lc \vf$ or $\Tc \vf$, using the following formulas: 
 \begin{align*}
  \varphi(\vx)= \frac{1}{2u_2}  \int_0^\infty \Lc \vf(\vx + s\vev_1)ds = -\frac{1}{2u_2}  \int_0^\infty \Tc \vf(\vx + s\vev_1)ds.
 \end{align*}
In this case, $\vf$ can also be reconstructed from $ \Mc\vf$ by solving the following second order partial differential equation for $\varphi$ (with appropriate initial/boundary conditions):
 \begin{align*}
        \left(1+2u_1^2\right)\frac{\partial^2 \varphi}{\partial x_1^2} + \left(u_1^2 - u_2^2\right) \frac{\partial^2 \varphi}{\partial x_2^2} &= -\frac{1}{u_2}\Xc_{\vev_2}\left(D_\vu D_\vv \Mc \vf \right)(\vx).
 \end{align*}
\end{itemize}
\end{theorem}
\begin{proof} \textbf{\textit{Part (a).}}
If $\vf = \D^2 \varphi$, then using \eqref{eq: L and M on df}  we get 
\begin{align*}
 \Lc \vf (\vx) &= - \left\langle \D \varphi(\vx) , \vu + \vv \right\rangle = -2u_2 \frac{\partial \varphi(\vx)}{\partial x_2},\\
 \Mc \vf (\vx) &= - \left\langle \D \varphi(\vx) , (\vu + \vv)^ \perp\right\rangle = 2u_2 \frac{\partial \varphi(\vx)}{\partial x_1}, \quad \mbox{ since } \delta^\perp (\D \phi) = 0.   
\end{align*}
Now, integrating $\Lc \vf(\vx)$ 
along $x_2$ and $\Mc\vf(\vx)$ along $x_1$ we get
  $$ \varphi(\vx)= \frac{1}{2u_2}  \int_0^\infty \Lc \vf(\vx + s\vev_2)ds=- \frac{1}{2u_2}  \int_0^\infty \Mc \vf(\vx + s\vev_1)ds.
  $$
So, the potential symmetric 2-tensor field $\vf = \D^2 \varphi$  can be recovered from  $\Lc\vf$ or $\Mc \vf$. 
%----------------------------------------------------

\vspace{3mm}
\noindent  \textbf{\textit{Part (b).}} If $\vf = (\D^\perp)^2 \varphi$, then from the relations in \eqref{eq: T and M on dperpf}  we get \\
\begin{align*}
 \Tc \vf (\vx) &=  - \left\langle \D^\perp \varphi(\vx) , (\vu + \vv)^\perp \right\rangle = - 2u_2 \frac{\partial \varphi(\vx)}{\partial x_2},  \\
 \Mc \vf (\vx) &= - \left\langle \D^\perp \varphi(\vx) , \vu + \vv\right\rangle = -2u_2 \frac{\partial \varphi(\vx)}{\partial x_1},  \quad \mbox{ since } \delta (\D^\perp \phi) = 0.
\end{align*}
Integrating  $ \Tc \vf (\vx)$
along $x_2$ and $\Mc \vf (\vx)$ along $x_1$ , we get 
$$ \varphi(\vx)= \frac{1}{2u_2}  \int_0^\infty \Tc \vf(\vx + s\vev_2)ds =  \frac{1}{2u_2}  \int_0^\infty \Mc \vf(\vx + s\vev_1)ds. $$
Hence,  a solenoidal symmetric 2-tensor field $\vf = (\D^\perp)^2 \varphi$ can be recovered from $\Tc\vf$ or $\Mc\vf$.

%----------------------------------------------------

\vspace{3mm}
\noindent \textbf{\textit{Part (c).}}  If $\vf = \D\D^\perp \varphi$, then  using the commutativity $ \D\D^\perp \varphi = \D^\perp\D \varphi $ and the expressions from  \eqref{eq: L and M on df}, \eqref{eq: T and M on dperpf} we get
\begin{align*}
\Lc \vf (\vx) &= - \left\langle \D^\perp \varphi(\vx) , \vu + \vv \right\rangle = -2u_2 \frac{\partial \varphi(\vx)}{\partial x_1},\\
\Tc \vf (\vx) &= - \left\langle \D \varphi(\vx) , (\vu + \vv)^\perp \right\rangle = 2u_2 \frac{\partial \varphi(\vx)}{\partial x_1}. 
\end{align*}
 Integrating  $ \Lc \vf (\vx)$
 and $ \Tc \vf (\vx)$  
 along $x_1$, we get  
 \begin{align*}
  \varphi(\vx)= \frac{1}{2u_2}  \int_0^\infty \Lc \vf(\vx + s\vev_1)ds = -\frac{1}{2u_2}  \int_0^\infty \Tc \vf(\vx + s\vev_1)ds   
 \end{align*}
 which gives the reconstruction of $\vf$ from $\Lc \vf$ and $\Tc \vf$. 
 
{Finally, using $\vf = \D\D^\perp \varphi$ in the second relation of equation \eqref{eq: T and M on dperpf} we get 
 \begin{align*}
\Mc \vf (\vx) &= - \left\langle \D \varphi(\vx) , \vu + \vv\right\rangle  -\frac{1}{2}\Vc\left(\delta \D \varphi\right)(\vx) = - 2u_2 \frac{\partial \varphi(\vx)}{\partial x_2} - \frac{1}{2}\Vc\left(\Delta \varphi\right)(\vx).
 \end{align*}
Following the approach used in part (a) of Theorem \ref{th:special tensor recovery 1}, the above equation can be modified into
\begin{align*}
    \left(1+2u_1^2\right)\frac{\partial^2 \varphi}{\partial x_1^2} + \left(u_1^2 - u_2^2\right) \frac{\partial^2 \varphi}{\partial x_2^2} = -\frac{1}{u_2}\Xc_{\vev_2}\left(D_\vu D_\vv \Mc \vf \right)(\vx).
\end{align*}
Thus, $\varphi$ satisfies a second order partial differential equation, which can be solved (with appropriate initial/boundary conditions) to get $\varphi$.}
\end{proof}

%%%%%%%%%%%%%%%%%%%%%%%%%%%%%%%%%%%%%%%%%%%%%%%%%%%
\section{Recovery of a tensor field $\vf$ from $\Lc \vf$, $\Tc \vf$, and $\Mc\vf$}\label{sec: reconstruction L, T, and M}
In this section, we show the full recovery of a symmetric 2-tensor field $\vf$ from the knowledge of its longitudinal, transverse, and mixed V-line transforms.
%------------------------------------
\begin{theorem}\label{th: full recovery}
Let $\vf\in C_c^2\left(S^2;D_1\right)$. 
Then $\vf$  can be explicitly determined from $\Lc \vf$, $\Tc \vf$, and $\Mc\vf$.
\end{theorem}
\begin{proof}{Recall, if $u_1 =  u_2$ then the result is already discussed in Corollary \ref{cor: u1 =u2 case}, therefore we will assume $u_1 \neq u_2$ in this proof.}
To derive the result, we first show that $f_{12}$ satisfies a boundary value problem for a second order PDE, the coefficients of which are explicitly defined by the available data (i.e. $\Lc \vf$, $\Tc \vf$, and $\Mc\vf$). This boundary value problem has a unique solution, which can be explicitly expressed in terms of the known data. We also derive an identity, which allows recovery of $f_{11}$ from the knowledge of $f_{12}$ and the available data. Finally, from the knowledge of $f_{11}$ and the pair of transforms $\Lc$ and $\Tc$, we recover $f_{22}$ using Theorem \ref{th:trace recovery}.

Recall $$D_\vu  = u_1 \partial_{x_1} +u_2 \partial_{x_2} \quad \mbox{ and } \quad D_\vv = -u_1 \partial_{x_1} +u_2 \partial_{x_2}.$$
 Using the definitions of $\Lc$, $\Tc$ and $\Mc$ and applying directional derivatives along $\vu$, $\vv$ we obtain 
\begin{align}\label{eq:DuDvLf}
-D_\vu D_\vv  \Lc \vf = D_\vv \langle \vf, \vu^2\rangle  + D_\vu \langle \vf, \vv^2\rangle 
= - 4 u_1^2 u_2 \partial_{x_1}  f_{12} +2 u_1^2 u_2 \partial_{x_2} f_{11} + 2u_2^3  \partial_{x_2}f_{22}.
\end{align}
\begin{align}\label{eq:DuDvTf}
-D_\vu D_\vv  \Tc \vf = D_\vv \langle \vf, \vu_\perp^2\rangle  + D_\vu \langle \vf, \vv_\perp^2\rangle = 4 u_1^2 u_2 \partial_{x_1}  f_{12} +2 u_2^3 \partial_{x_2} f_{11} + 2u_1^2u_2  \partial_{x_2}f_{22}.
\end{align}
\begin{align}\label{eq:DuDvMf}
-D_\vu D_\vv  \Mc \vf &= D_\vv \langle \vf,\vu \odot \vu_\perp\rangle  + D_\vu \langle \vf,\vv \odot \vv_\perp\rangle \nonumber\\ & = 2 u_1^2 u_2 \partial_{x_1}  f_{11} - 2 u_1^2 u_2 \partial_{x_1} f_{22} + 2u_2 (u_1^2 -u_2^2)\partial_{x_2} f_{12}.
\end{align}
Multiplying  equation \eqref{eq:DuDvLf} by  $u_1^2$ and \eqref{eq:DuDvTf} by $u_2^2$, and then subtracting one from another we obtain
\begin{align}\label{eq:f_11}
u_2^2 D_\vu D_\vv  \Tc \vf - u_1^2 D_\vu D_\vv  \Lc \vf =  - 4 u_1^2 u_2 \partial_{x_1}  f_{12} + 2 u_2(u_1^2 -u_2^2) \partial_{x_2} f_{11}.
\end{align}
Differentiating \eqref{eq:DuDvTf} with respect to $x_1$,  \eqref{eq:DuDvMf} with respect to $x_2$, and then adding them we get
\begin{align}\label{eq:f_11,b=0}
-\partial_{x_1} D_\vu D_\vv  \Tc \vf-\partial_{x_2} D_\vu D_\vv  \Mc \vf &= 2 u_2 \partial_{x_1}\partial_{x_2} f_{11} + 2 u_2 \left[2u_1^2 \partial_{x_1}^2 + (u_1^2 -u_2^2)\partial_{x_2}^2\right]f_{12}.
\end{align}
Multiplying both sides of equation (\ref{eq:f_11,b=0}) by $(u_1^2 -u_2^2)$ we get
\begin{align*}
   &- (u_1^2 -u_2^2)\left(\partial_{x_1} D_\vu D_\vv  \Tc \vf +\partial_{x_2} D_\vu D_\vv  \Mc \vf \right)\\
   & \qquad \qquad = 2 u_2(u_1^2 -u_2^2) \partial_{x_1}\partial_{x_2} f_{11} 
  + 2 u_2(u_1^2 -u_2^2) \left[2u_1^2 \partial_{x_1}^2 + (u_1^2-u_2^2)\partial_{x_2}^2\right]f_{12}.
\end{align*}
Using the expression for $2 u_2(u_1^2 -u_2^2) \partial_{x_2} f_{11}$ from equation \eqref{eq:f_11} and simplifying further, we get
\begin{align}\label{eq:f_12}
&2 u_1^2[1+(u_1^2 -u_2^2)] \partial_{x_1}^2 f_{12} + ( u_1^2-u_2^2)^2 \partial_{x_2}^2 f_{12}\nonumber\\
&= -\frac{1}{2u_2}D_\vu D_\vv\left[u_1^2 \partial_{x_1}  \left( \Tc \vf-\Lc \vf\right)+ (u_1^2 -u_2^2)\partial_{x_2} \Mc \vf\right]. 
\end{align}

\noindent Equation \eqref{eq:f_12} is an elliptic PDE with respect to $f_{12}$ and can be written as follows:
$$ a \partial_{x_1}^2 f_{12}+ b\partial_{x_2}^2 f_{12} = -g, $$
where  $a=2 u_1^2\left[1+(u_1^2 -u_2^2)\right ]  > 0 $, ~ $b=\left( u_1^2-u_2^2\right)^{2} >  0$ (since $u_1 \neq u_2$), and 
$$g = \frac{1}{2u_2}\left[u_1^2 \partial_{x_1} D_\vu D_\vv \left( \Tc \vf-\Lc \vf\right)+(u_1^2 -u_2^2)\partial_{x_2} D_\vu D_\vv  \Mc \vf\right] $$ is a known quantity in terms of the given data. Therefore, $f_{12}$ can be found by solving the following elliptic boundary value problem:
\begin{align}\label{eq:elliptic equation for f12}
    \left\{\begin{array}{rll}
   a \partial_{x_1}^2 f_{12}+ b\partial_{x_2}^2 f_{12}&= -g      & \mbox{ in } D_1,  \\
     f_{12} &= 0    & \mbox{ on } \partial D_1.
    \end{array}\right.
\end{align}
Once $f_{12}$ is known, we use \eqref{eq:f_11} to get $f_{11}$. Finally, using the knowledge of the trace $f_{11}+f_{22}$, we obtain $f_{22}$.
\end{proof}

%%%%%%%%%%%%%%%%%%%%%%%%%%%%%%%%%%%%%%%%%%%%%%%%%%%%%%%%
\section{Reconstruction of a tensor field using integral moments}\label{sec:moments}
It is known from \cite[Theorem 2.17.2]{Sharafutdinov_Book} that a symmetric $m$-tensor field is determined by its first $(m+1)$-integral moment transforms \textit{integrating along straight lines}. In this section we show that a similar statement does not hold in the V-line setup. We prove this by showing that $\Lc^k \vf$ can be expressed in terms of  $\Lc^{k-1} \vf$ and $\Lc^{k-2} \vf$ for $k \geq 2$. In other words, assuming the knowledge of $\Lc=\Lc^0$ and $\Lc^1$, the moments of higher order $k\geq 2$ do not provide any new information. We also show that $\vf$ can be recovered from the knowledge of $\Lc \vf$, $\Lc^1 \vf$, and $\Tc \vf$. Similar statements also hold for the transverse and mixed VLTs and their moment transforms. We state them along with the corresponding analogs for the longitudinal transform. Since the proofs for all these cases are identical, we only discuss the proofs for  the longitudinal case.

\begin{lem}\label{lem: recurrence relation for L}
   Let $\vf\in C_c^2\left(S^2;D_1\right)$. Then we have the following recurrence relations:
\begin{align}
     - D_\vu D_\vv\Lc^k\vf &= k(k-1)\Lc^{k-2}\vf+ k (D_\vu +D_\vv)\Lc^{k-1}\vf, \;\;\;\;k\ge2, \label{eq:iterative relation}\\
     - D_\vu D_\vv\Tc^k\vf &= k(k-1)\Tc^{k-2}\vf+ k (D_\vu +D_\vv)\Tc^{k-1}\vf, \;\;\;\;k\ge2, \label{eq:iterative relation for T}  \\
 - D_\vu D_\vv\Mc^k\vf&= k(k-1)\Mc^{k-2}\vf+ k (D_\vu +D_\vv)\Mc^{k-1}\vf, \;\;\;\;k\ge2 \label{eq:iterative relation M}.   
   \end{align}
\end{lem}
\begin{proof}
Using integration by parts, we get the following formulas:
\begin{align}\label{eq:integration by parts}
 D_\vu\mathcal{X}^1_{\vu}\left( \left\langle \vf , \vu^2 \right\rangle \right)=- \mathcal{X}_{\vu}\left( \left\langle \vf , \vu^2 \right\rangle \right) \quad \mbox{ and } \quad  D_\vu \mathcal{X}^2_{\vu}\left( \left\langle \vf , \vu^2 \right\rangle \right)= -2 \mathcal{X}^1_{\vu}\left( \left\langle \vf , \vu^2 \right\rangle \right).  
\end{align}
Applying $D_\vv$ to the first relation and using the identity $D_\vv= 2 u_2 \partial_{x_2}-D_\vu$, we get 
$$D_\vv D_\vu\mathcal{X}^1_{\vu}\left( \left\langle \vf , \vu^2 \right\rangle \right)= -D_\vv\mathcal{X}_{\vu}\left( \left\langle \vf , \vu^2 \right\rangle \right) = - \left\langle \vf , \vu^2 \right\rangle  - 2 u_2 \partial_{x_2}\mathcal{X}_{\vu}\left( \left\langle \vf , \vu^2 \right\rangle \right). $$
Similarly, we have $$D_\vu D_\vv\mathcal{X}^1_{\vv}\left( \left\langle \vf , \vv^2 \right\rangle \right)= -D_\vu\mathcal{X}_{\vv}\left( \left\langle \vf , \vv^2 \right\rangle \right) = -\left\langle \vf , \vv^2 \right\rangle -  2 u_2 \partial_{x_2}\mathcal{X}_{\vv}\left( \left\langle \vf , \vv^2 \right\rangle \right). $$
Combining the above two equations, we have the following relation:
\begin{align}\label{eq:L1 in terms of L}
  -D_\vu D_\vv\Lc^1\vf =  \left\langle \vf , \vu^2 \right\rangle + \left\langle \vf , \vv^2 \right\rangle + 2 u_2 \partial_{x_2}\Lc\vf = 2 \left(u_1^2f_{11}+u_2^2f_{22}\right)+ (D_\vu + D_\vv)\Lc\vf.
\end{align}
Next, we consider the second identity in \eqref{eq:integration by parts} and use a similar analysis as above to get the following identities for the second integral moment of $\Lc$:
$$ D_\vv D_\vu \mathcal{X}^2_{\vu}\left( \left\langle \vf , \vu^2 \right\rangle \right)= D_\vv\left(-2 \mathcal{X}^1_{\vu}\left( \left\langle \vf , \vu^2 \right\rangle \right)\right)= -2\mathcal{X}_{\vu}\left( \left\langle \vf , \vu^2 \right\rangle \right) -4 u_2 \partial_{x_2}\mathcal{X}^1_{\vu}\left( \left\langle \vf , \vu^2 \right\rangle \right), $$
$$ D_\vu D_\vv \mathcal{X}^2_{\vv}\left( \left\langle \vf , \vv^2 \right\rangle \right)= D_\vu\left(-2 \mathcal{X}^1_{\vv}\left( \left\langle \vf , \vv^2 \right\rangle \right)\right)= -2\mathcal{X}_{\vv}\left( \left\langle \vf , \vv^2 \right\rangle \right) -4 u_2 \partial_{x_2}\mathcal{X}^1_{\vv}\left( \left\langle \vf , \vv^2 \right\rangle \right). $$
Combining these two relations, we have 
$$ -D_\vu D_\vv\Lc^2\vf= 2\Lc\vf + 4 u_2 \partial_{x_2}\Lc^1\vf= 2\Lc\vf + 2(D_\vu + D_\vv)\Lc^1\vf.$$
Repeating the same idea for $k \geq 2$ we arrive at the recurrence relation (\ref{eq:iterative relation}). The relations \eqref{eq:iterative relation for T} and \eqref{eq:iterative relation M} are proved similarly.
\end{proof}
Relation (\ref{eq:iterative relation}) tells us that the higher order  moments $\Lc^k \vf$ ($k \geq 2$) of $\vf$ can be generated from the knowledge of $\Lc \vf$ and $\Lc^1 \vf$. Therefore,  $\Lc^k \vf$ ($k \geq 2$) is not giving any new information, hence $\vf$ cannot be reconstructed only from its integral moment transforms. 

For example, let $\displaystyle \vf(\vx) = \begin{pmatrix}
      u_2^2 & 0\\
        0 & -u_1^2
\end{pmatrix} \varphi(\vx)$, where 
$$ \displaystyle  \varphi(\vx)= \begin{cases}
  \exp\left(-\frac{1}{1-\lvert\vx\rvert^{2}}\right), &\mbox{for } \lvert\vx\rvert<1 , \\  
0, & \mbox{for } \lvert\vx\rvert\ge1.
    \end{cases} 
$$  
Then $\Lc(\vf)(\vx) \equiv  \Lc^1(\vf)(\vx)\equiv0$, i.e.  $\displaystyle \vf \in \ker(\Lc) \cap \ker(\Lc^1) $. Therefore, $\vf \not\equiv0$ cannot be recovered just from the knowledge of $\Lc(\vf)(\vx)$ and $ \Lc^1(\vf)(\vx) $.

However, as we discuss below, one can expect to recover some partial information about $\vf$ from $\Lc\vf$ and $\Lc^1\vf$, as well as a full recovery of the field with some additional data. 

%--------------------------------------------------
\begin{theorem}\label{th: full recovery L, L1, T}
If $u_1\ne u_2$, then $\vf\in C_c^2\left(S^2;D_1\right)$ 
can be recovered from $\Lc \vf$, $\Lc^1 \vf$, and $\Tc \vf$.
\end{theorem}
%--------------------------------------------------
\begin{proof}
Recall that from equation \eqref{eq: Lf} we have
\begin{align}
\Lc \vf (\vx) &= u_1^2 \Vc(f_{11})(\vx) + 2u_1u_2 \Vc^-(f_{12})(\vx) + u_2^2 \Vc(f_{22})(\vx).\label{eq: Lf-n}
\end{align}
Using this relation and the inversion formula for $\Vc$, equation \eqref{eq:L1 in terms of L} can be re-written as
\begin{align}
    -D_\vu D_\vv \Lc^1 \vf &= u_1^2\left( \frac{D_\vu D_\vv\mathcal{X}_{\vev_2}}{u_2}+D_\vu + D_\vv\right) \Vc(f_{11}) + 2u_1u_2\left(D_\vu+ D_\vv \right) \Vc^-(f_{12})\nonumber\\
& \quad  + u_2^2 \left( \frac{D_\vu D_\vv \mathcal{X}_{\vev_2}}{u_2}+D_\vu + D_\vv\right)\Vc(f_{22})\label{eq: L1}.
\end{align}
Applying the operator $\left( \displaystyle\frac{D_\vu D_\vv \mathcal{X}_{\vev_2}}{u_2}+D_\vu + D_\vv\right)$ to equation  \eqref{eq: Lf-n} and adding it to the negative of equation \eqref{eq: L1},  we obtain
\begin{align}
 \left( \frac{D_\vu D_\vv \Xc_{\vev_2}}{u_2}+ D_\vu + D_\vv\right)\Lc\vf (\vx)  + D_\vu D_\vv \Lc^1\vf (\vx) &= 2u_1 D_\vu D_\vv\Xc_{\vev_2} \Vc^-(f_{12})(\vx) \label{eq:sign}.
 \end{align}
Observe that $$D_\vu D_\vv\Vc^-(f_{12})=-D_\vv f_{12}+D_\vu f_{12} = 2u_1 \partial_{x_1}f_{12}.$$ 
With this, the above equation simplifies to 
\begin{align*}
 \left( \frac{D_\vu D_\vv \Xc_{\vev_2}}{u_2}+ D_\vu + D_\vv\right)\Lc\vf (\vx)  + D_\vu D_\vv \Lc^1\vf (\vx) &= 4u_1^2\partial_{\x_1}\Xc_{\vev_2}(f_{12})(\vx),
 \end{align*}
 which yields
 \begin{align}\label{eq: f12 from moments}
f_{12}(\vx) = \frac{1}{4u_1^2} \partial_{x_2}\Xc_{\vev_1}\left[\left( \frac{D_\vu D_\vv \Xc_{\vev_2}}{u_2}+ D_\vu + D_\vv\right)\Lc\vf (\vx)  + D_\vu D_\vv \Lc^1\vf (\vx)\right].
 \end{align}
Therefore, $f_{12}$ can be expressed explicitly in terms of $\Lc \vf$ and $\Lc^1 \vf$.

Furthermore, equation \eqref{eq:L1 in terms of L} gives the following:
\begin{align}\label{eq:modified trace of f}
\left(u_1^2f_{11}+u_2^2f_{22}\right)(\vx) =  - \frac{1}{2} \left[D_\vu D_\vv\Lc^1\vf (\vx) +  (D_\vu + D_\vv)\Lc\vf (\vx)\right].
\end{align}
If we also have the knowledge of $\Tc \vf$, in addition to $\Lc \vf$ and $\Lc^1 \vf$, then the trace of $\vf$ is also known from Theorem \ref{th:trace recovery}, that is, we have  
\begin{align*}
    \left(f_{11} + f_{22}\right)(\vx) = \frac{1}{2 u_2} D_\vu D_\vv  \int_{0}^{\infty}\left(\Lc \vf + \Tc \vf \right)(\vx+ t \vev_2)\,dt.
\end{align*}
Since $u_1\ne u_2$, one can solve the above equation together with \eqref{eq:modified trace of f} to recover $f_{11}$ and $f_{22}$. 
\end{proof}

\vspace{1mm}

Note from equation \eqref{eq:L1 in terms of L} that when $u_1=u_2$, the knowledge of $\Lc\vf$ and $\operatorname{trace}(\vf)=f_{11}+f_{22}$ (or, equivalently, $\Lc\vf$ and $\Tc\vf$) uniquely determines $\Lc^1\vf$. Therefore, the exclusion of the case $u_1=u_2$ from the previous theorem is not an artifact of the presented method of proof.

\vspace{1mm}

\begin{theorem}\label{th: full recovery from  various moments }
If $u_1\ne u_2$, then $\vf\in C_c^2\left(S^2;D_1\right)$ 
can be recovered from any of the following three combinations:
\begin{enumerate}
    \item $\Tc \vf$, $\Tc^1 \vf$, and $\Lc\vf$.
    \item $\Mc \vf$, $\Mc^1 \vf$, and $\Lc\vf$.
    \item $\Mc \vf$, $\Mc^1 \vf$, and $\Tc\vf$.
\end{enumerate}
\end{theorem}
\begin{proof}
Similar to the longitudinal case above. 
\end{proof}
%------------------------------------------------------
\vspace{1mm}

\begin{theorem}\label{th: full recovery L, L1, M}
Any tensor field $\vf\in C_c^2\left(S^2;D_1\right)$  can be recovered from $\Lc \vf$, $\Lc^1 \vf$, and $\Mc\vf$.
\end{theorem}
%-------------------------------------------------------
\begin{proof}
The relation \eqref{eq: f12 from moments} of the previous Theorem \ref{th: full recovery L, L1, T} gives a method to recover $f_{12}$ from $\Lc \vf$ and $\Lc^1 \vf$. Also, we have the following relation between $f_{11}$ and $f_{22}$ (from equation \eqref{eq:modified trace of f}):
    \begin{align*}
        \left(u_1^2f_{11}+u_2^2f_{22}\right)(\vx) &=  - \frac{1}{2} \left(D_\vu D_\vv\Lc^1\vf (\vx) +  (D_\vu + D_\vv)\Lc\vf (\vx)\right).
    \end{align*}
Recall from equation (\ref{eq:Mf}), $\Mc \vf (\vx) = -u_1u_2 \Vc^{-}(f_{11})(\vx) + (u_1^2 -u^2_2) \Vc(f_{12})(\vx) + u_1u_2 \Vc^{-}(f_{22})(\vx)$, and we already know $f_{12}$. Thus, the right-hand side of the equation below is known:
\begin{align*}
    u_1u_2\Vc^{-}\left(f_{11} - f_{22}\right)(\vx) =  (u_1^2 -u_2^2)\Vc(f_{12})(\vx) - \Mc \vf (\vx).
\end{align*}
This relation gives $f_{11} - f_{22}$ explicitly in terms of the known data, which we combine with the relation above for  $u_1^2f_{11}+u_2^2f_{22}$ to recover $f_{11}$ and $f_{22}$. 
\end{proof}
%---------
\begin{theorem}\label{th: full recovery T, T1, M}
Any tensor field $\vf\in C_c^2\left(S^2; D_1\right)$ can be recovered from $\Tc \vf$, $\Tc^1 \vf$, and $\Mc\vf$.
\end{theorem}
\begin{proof}
Similar to the longitudinal case above. 
\end{proof}
%---------------------------------
%%%%%%%%%%%%%%%%%%%%%%%%%%%%%%%%%%%%%%%%%%%%%%%%%%%%%%%%

\section{Full recovery of a tensor field from its star transform}\label{sec:Star}
In this section, we study the star transform for symmetric 2-tensor fields and derive its inversion formula. Throughout the section, we identify a symmetric 2-tensor $\vg = \left(g_{ij}\right)$ with the vector $\vg = \left(g_{11}, g_{12},  g_{22}\right)\in\Rb^3$. In particular, we make no distinction between a symmetric 2-tensor field $\vf(\vx) = (f_{ij}(\vx))$ and the corresponding vector field $\vf = (f_{11}(\vx), f_{12}(\vx),  f_{22}(\vx))$. Similarly, the tensors $\vu^2$ and $\vu \odot \vv$ are identified with the vectors $\left(u_1^2, u_1 u_2, u_2^2\right)$ and $\left(u_1v_1, \frac{1}{2}(u_1 v_2 + u_2 v_1), u_2v_2\right)$. With these adaptations, we define the star transform of a symmetric 2-tensor field as follows. 
%---------------------------------------
\begin{defn}
Let $\vf\in C_c^2\left(S^2;\Rb^2\right)$,
and let $\vgamma_{1},\vgamma_{2},...,\vgamma_{m}$ be distinct unit vectors in $\mathbb{R}^{2}$. The \textbf{star transform} of $\vf$ is defined as
\begin{align}
\Sc\vf &=  \sum_{i=1}^{m} c_{i} {\Xc}_{{\vgamma}_{i}} 
\begin{bmatrix}
 \vf\cdot {{\vgamma}_{i}}^{2}\\
 \vf\cdot {{\vgamma}_{i} \odot {\vgamma}_{i}^{\perp} }\\
 \vf\cdot ({\vgamma}_{i}^{\perp})^{2}
\end{bmatrix} ,
\end{align} 
where $c_{1}, c_{2},\dots , c_{m}$ are non-zero constants in $\mathbb{R}$.
\end{defn}
 %---------------------------------------------------
 Note that the star transform data contains longitudinal, transverse, and mixed components. 
 For a function $h\in C^2_c(\Rb^2)$ we denote by ${\Rc} h(\vxi,s)$ the Radon transform of $h$, i.e. the integral of $h$ along the line $l=\{\vx\in \Rb^2: \vx\cdot\vxi=s\},$ where $\vxi$ is the unit normal vector to the line, and $s$ is the signed distance from the origin. From the Lemmas 1 and 2 in \cite{Amb_Lat_star} we have the following identity:
 \begin{equation}\label{eq: radon identity}
  \frac{d}{ds} \Rc ({\Xc}_{{\vgamma}_{i}} h) (\vxi,s)= \frac{-1}{\vxi \cdot {\vgamma}_{i}} \Rc h (\vxi,s).
  \end{equation}

%------------------------------------------------
\begin{defn}
Consider the  star transform $\Sc \vf$  of a symmetric 2-tensor field $\vf$ with branches along directions $\vgamma_{1},\vgamma_{2},...,\vgamma_{m}$. We call
$$ \Zc_{1}= \cup_{i=1}^{m}\{ \vxi : \vxi \cdot {\vgamma}_{i} =0\} $$ the set of singular directions of type 1 for $\Sc$.
\end{defn}
%-----------------------------------------------
Now let us define three vectors in $\Rb^3$ which will be important for further calculations. For $\vxi \in \mathbb{S}^{1}\setminus \Zc_{1}$, we define 
\begin{align}
\vgamma(\vxi) = - \sum_{i=1}^{m} \frac{c_{i}{\vgamma}_{i}^{2}}{\vxi \cdot {\vgamma}_{i}},\qquad
\vgamma^\dag(\vxi) = - \sum_{i=1}^{m} \frac{c_{i}{\vgamma}_{i}\odot \vgamma_i^\perp}{\vxi \cdot {\vgamma}_{i}},\qquad
\vgamma^\perp(\vxi) = - \sum_{i=1}^{m} \frac{c_{i}({\vgamma_i}^\perp)^2 }{\vxi \cdot {\vgamma}_{i}}\label{eq: gamma_xi perp}.
\end{align}
%-----------------------------------------------
\begin{defn}
We call $$\Zc_{2} = \left\{ \vxi :  \vgamma^\dag(\vxi) = 0\right\}\ \bigcup \ \left\{ \vxi :\sum_{i=1}^m\frac{c_i}{\vxi \cdot \vgamma_i} = 0 \right\}$$ 
the set of singular directions of type 2 for $\Sc$.
\end{defn}
%--------------------------------------------------

\begin{lem}\label{lem:zero sets of vgamma}
    The zero sets of $\vgamma(\vxi)$, $\vgamma^\dag(\vxi)$, and $\vgamma^\perp(\vxi)$ satisfy the following relations: 
    $$
    \{ \vxi :  \vgamma(\vxi) = 0 \} = \{ \vxi :  \vgamma^\perp(\vxi) = 0 \} \subset  \{ \vxi :  \vgamma^\dag(\vxi) = 0 \}.
    $$ 
\end{lem}
%------------------------------------------------
\begin{proof}
First, let us rewrite $\vgamma(\vxi)$, $\vgamma^\dag(\vxi)$, and $\vgamma^\perp(\vxi)$ with a common denominator as follows:
\begin{align*}
\vgamma(\vxi) &= - \sum_{i=1}^{m} \frac{c_{i}{\vgamma}_{i}^{2}}{\vxi \cdot {\vgamma}_{i}}  = -\sum_{i=1}^{m}\frac{\left(\prod_{j\ne i}\vxi \cdot {\vgamma}_{j}\right)c_{i}\vgamma_{i}^2}{\prod_{j=1}^{m} \vxi \cdot {\vgamma}_{j}} 
= -\frac{\Pc(\vxi)}{\kappa(\vxi)} = -\frac{\left(\Pc_{11}(\vxi),\Pc_{12}(\vxi),\Pc_{22}(\vxi)\right)}{\kappa(\vxi)},\\
\vgamma^\dag(\vxi) &= - \sum_{i=1}^{m} \frac{c_{i}{\vgamma}_{i}\odot \vgamma_i^\perp}{\vxi \cdot {\vgamma}_{i}} = -\sum_{i=1}^{m}\frac{\left(\prod_{j\ne i}\vxi \cdot {\vgamma}_{j}\right)c_{i}\vgamma_{i}\odot \vgamma_i^\perp}{\prod_{j=1}^{m} \vxi \cdot {\vgamma}_{j}} 
= -\frac{\Pc^\dag(\vxi)}{\kappa(\vxi)} = -\frac{\left(\Pc^\dag_{11}(\vxi),\Pc^\dag_{12}(\vxi),\Pc^\dag_{22}(\vxi)\right)}{\kappa(\vxi)},\\
\vgamma^\perp(\vxi) &= - \sum_{i=1}^{m} \frac{c_{i}({\vgamma_i}^\perp)^2 }{\vxi \cdot {\vgamma}_{i}} = -\sum_{i=1}^{m}\frac{\left(\prod_{j\ne i}\vxi \cdot {\vgamma}_{j}\right)c_{i} (\vgamma_i^\perp)^2}{\prod_{j=1}^{m} \vxi \cdot {\vgamma}_{j}} 
= -\frac{\Pc^\perp(\vxi)}{\kappa(\vxi)}= -\frac{\left(\Pc^\perp_{11}(\vxi),\Pc^\perp_{12}(\vxi),\Pc^\perp_{22}(\vxi)\right)}{\kappa(\vxi)},
\end{align*}
where $\displaystyle \kappa(\vxi) = \prod^m_{j=1}\vxi \cdot \vgamma_{j}$, and $\Pc(\vxi)$, $\Pc^\dag(\vxi)$, $\Pc^\perp(\vxi)$ are the numerators of the corresponding fractions. 

\noindent For $i  = 1, \dots , m$, take $\vgamma_i = \left(\gamma_{i1}, \gamma_{i2}\right)$ and  $\vgamma_i^\perp = \left(-\gamma_{i2}, \gamma_{i1}\right)$ to get the following expressions for the components of $\Pc(\vxi)$, $\Pc^\dag(\vxi)$, and $\Pc^\perp(\vxi)$:
\begin{align*}
\Pc_{11}(\vxi) &= \sum_{i=1}^m \kappa_i(\vxi)  c_i \gamma_{i1}^2, \quad    \Pc_{12}(\vxi) = \sum_{i=1}^m \kappa_i(\vxi)  c_i\gamma_{i1}\gamma_{i2}, \quad \mbox{and} \quad \Pc_{22}(\vxi) = \sum_{i=1}^m \kappa_i(\vxi)  c_i \gamma_{i2}^2,\\
\Pc^\dag_{11}(\vxi) &= -\sum_{i=1}^m \kappa_i(\vxi)  c_i\gamma_{i1}\gamma_{i2}, \quad    \Pc^\dag_{12}(\vxi) = \sum_{i=1}^m \kappa_i(\vxi)  c_i \frac{\left(\gamma_{i1}^2 - \gamma_{i2}^2\right)}{2}, \quad \mbox{and} \quad \Pc^\dag_{22}(\vxi) = \sum_{i=1}^m \kappa_i(\vxi)  c_i\gamma_{i1}\gamma_{i2}, \\
\Pc_{11}^\perp(\vxi) &= \sum_{i=1}^m \kappa_i(\vxi)  c_i \gamma_{i2}^2, \quad    \Pc^\perp_{12}(\vxi) = - \sum_{i=1}^m \kappa_i(\vxi)  c_i\gamma_{i1}\gamma_{i2}, \quad \mbox{and} \quad \Pc^\perp_{22}(\vxi) = \sum_{i=1}^m \kappa_i(\vxi)  c_i \gamma_{i1}^2,
\end{align*}
where $\displaystyle \kappa_i(\vxi) = \prod_{j\ne i}\vxi \cdot {\vgamma}_{j}$. The above expressions imply
\begin{align}
  \Pc_{11}(\vxi) &= \Pc^\perp_{22}(\vxi), \quad  \quad  \Pc_{12}(\vxi) = -\Pc^\perp_{12}(\vxi), \  \mbox{ and }\   \Pc_{22}(\vxi) = \Pc^\perp_{11}(\vxi), \label{eq:relation between Pii and Pii perp}\\
    \Pc_{11}^\dag(\vxi) &= - \Pc_{12}(\vxi), \quad   \Pc^\dag_{12}(\vxi) = \frac{1}{2} \left[\Pc_{11}(\vxi) - \Pc_{22}(\vxi)\right], \ \mbox{ and } \  \Pc^\dag_{22}(\vxi) = \Pc_{12}(\vxi). \label{eq:relation between Pii and Pii dag}
\end{align}
From these relations, we have $\displaystyle \vgamma(\vxi) =-\frac{1}{\kappa(\vxi)}\left(\Pc_{11}(\vxi),\Pc_{12}(\vxi),\Pc_{22}(\vxi)\right) = 0$ if and only if $\displaystyle \vgamma^\perp(\vxi) =  -\frac{1}{\kappa(\vxi)}\left(\Pc_{22}(\vxi),-\Pc_{12}(\vxi),\Pc_{11}(\vxi)\right) = 0$, which implies the following two sets are equal:
$$\{ \vxi :  \vgamma(\vxi) = 0 \} = \{ \vxi :  \vgamma^\perp(\vxi) = 0 \}.
    $$ 
This proves the equality part of the lemma. To obtain the inclusion relation and complete the proof, let us take $\vxi \in \Sb^1$ such that $\vgamma(\vxi) = 0$, that is, $\left(\Pc_{11}(\vxi),\Pc_{12}(\vxi),\Pc_{22}(\vxi)\right) = (0,0,0)$. Then, equation \eqref{eq:relation between Pii and Pii dag} implies that $\vgamma^\dag(\vxi) =0 $, which yields the required containment $$\{ \vxi :  \vgamma(\vxi) = 0 \}   \subset \{ \vxi :  \vgamma^\dag(\vxi) = 0 \}.$$ 
An inclusion relation in the opposite direction in the statement above is not true. For example, if $\vgamma_1 = (1, 0)$, $\vgamma_2 = (0, 1)$, $ \displaystyle \vxi = \left(\frac{1}{\sqrt{2}}, \frac{1}{\sqrt{2}}\right)$, and $c_1 = c_2 = 1$, then one can verify that
$\vgamma^\dag (\vxi) = (0,0,0)$ and $\vgamma (\vxi) = -\sqrt{2}(1, 0, 1).$
\end{proof}
%%%%%%%%%%%%%%%%%%%%%%%%%%%%%%%%%%%%%%%%%%%%%%%%%%%%%%% 

\begin{lem}
For $\vxi \notin \Zc_2$, the vectors $\vgamma(\vxi)$, $\vgamma^\dag(\vxi)$, and $\vgamma^\perp(\vxi)$ are linearly independent in $\Rb^3$, therefore the following matrix is invertible:
    \begin{align*}
       \Qc(\vxi) =  \begin{bmatrix}
             {\vgamma}(\vxi)\\
 {\vgamma}^{\dag}(\vxi) \\
 {\vgamma}^{\perp}(\vxi)
        \end{bmatrix}.
    \end{align*}
\end{lem}
%--------------------------------------------------
\begin{proof}
We prove $\Qc(\vxi)$ is invertible by showing that $\operatorname{det} \Qc\ne0$ for $\vxi \notin \Zc_2$. By \eqref{eq:relation between Pii and Pii perp} and \eqref{eq:relation between Pii and Pii dag}
\begin{align*}
\Qc(\vxi) =  \begin{bmatrix}
{\vgamma}(\vxi)\\
 {\vgamma}^{\dag}(\vxi) \\
 {\vgamma}^{\perp}(\vxi)
        \end{bmatrix} &= -\frac{1}{\kappa (\vxi)}\begin{bmatrix}
\Pc_{11}(\vxi) & \Pc_{12}(\vxi) & \Pc_{22}(\vxi)\\
\Pc^\dag_{11}(\vxi) & \Pc^\dag_{12}(\vxi) & \Pc^\dag_{22}(\vxi)\\
\Pc^\perp_{11}(\vxi) & \Pc^\perp_{12}(\vxi) & \Pc^\perp_{22}(\vxi)
        \end{bmatrix}  \\
        &= -\frac{1}{\kappa(\vxi)}\begin{bmatrix}
\Pc_{11}(\vxi) & \Pc_{12}(\vxi) & \Pc_{22}(\vxi)\\
-\Pc_{12}(\vxi) & \frac{1}{2}\left( \Pc_{11} - \Pc_{22}\right)(\vxi) & \Pc_{12}(\vxi)\\
\Pc_{22}(\vxi) & -\Pc_{12}(\vxi) & \Pc_{11}(\vxi)
        \end{bmatrix}. 
\end{align*}
Therefore, 
\begin{align*} \det\Qc(\vxi)\ &= -\frac{1}{2\kappa^3(\vxi)}\left\{4\Pc^2_{12}(\vxi) + \left[\Pc_{11}(\vxi) -\Pc_{22}(\vxi)\right]^2\right\}\left[\Pc_{11}(\vxi) + \Pc_{22}(\vxi)\right]\\
&= -\frac{1}{2\kappa^3(\vxi)}\left\{4\Pc^2_{12}(\vxi) + \left[\Pc_{11}(\vxi) -\Pc_{22}(\vxi)\right]^2\right\}\sum_{i=1}^m  \kappa_i(\vxi) c_i,
\end{align*}
where in the last line, we used the identity $\displaystyle \Pc_{11}(\vxi) + \Pc_{22}(\vxi) =  \sum_{i=1}^m  \kappa_i(\vxi) c_i$. 

For $\vxi \notin \Zc_2$, both terms $\displaystyle 4\Pc^2_{12}(\vxi) + \left[\Pc_{11}(\vxi) -\Pc_{22}(\vxi)\right]^2$ and $\displaystyle \sum_{i=1}^m  \kappa_i(\vxi) c_i$, appearing in the right-hand side of the above equation, are non-zero. Therefore $\displaystyle     \det\Qc(\vxi)\neq 0$.
\end{proof}

%%%%%%%%%%%%%%%%%%%%%%%%%%%%%%%%%%%%%%%%%%%%%%%%%%%%%%%  

%----------------------------------------------- 
\begin{theorem}\label{th:inversion of star transform S}
 Let $\vf\in C_c^2\left(S^2;D_1\right)$, and $\vgamma_{1},\vgamma_{2}, \dots ,\vgamma_{m}$ be the branch directions of the star transform. Then for any $\vxi \in \Sb^{1}\setminus(\Zc_{1}\cup\Zc_{2}) $ and any $ s\in \Rb$ we have 
\begin{equation} \label{eq: star inversion}
\left[\Qc(\vxi)\right]^{-1}\frac{d}{ds} \Rc ({\Sc}\vf) (\vxi,s)=  \Rc \vf (\vxi,s)
    \end{equation}
where $\Rc \vf$ is the component-wise Radon transform of $\vf$.
\end{theorem}
%-----------------------------------------------
 \begin{proof}
     Using \eqref{eq: radon identity} we have 
     \begin{equation}
      \frac{d}{ds} \Rc ({\Sc}\vf) =  \sum_{i=1}^{m} c_{i} \frac{d}{ds}{\Rc}{\Xc}_{{\vgamma}_{i}} 
\begin{bmatrix}
 \vf\cdot {{\vgamma}_{i}}^{2}\\
 \vf\cdot {{\vgamma}_{i}\odot \vgamma_i^\perp }\\
 \vf\cdot ({\vgamma}_{i}^{\perp})^{2}
\end{bmatrix}  = - \sum_{i=1}^{m} \frac{c_{i}}{\vxi \cdot \vgamma_{i}}{\Rc}\begin{bmatrix}
 \vf\cdot {{\vgamma}_{i}}^{2}\\
 \vf\cdot {{\vgamma}_{i}\odot \vgamma_i^\perp }\\
 \vf\cdot ({\vgamma}_{i}^{\perp})^{2}
\end{bmatrix}.
     \end{equation}
      Using the linearity of both $\Rc$ and the inner product, we obtain 
      \begin{equation}
         \frac{d}{ds} \Rc ({\Sc}\vf) = \begin{bmatrix}
 \Rc\vf\cdot {\vgamma}(\vxi)\\
 \Rc\vf\cdot {\vgamma}^{\dag}(\vxi)\\
 \Rc\vf\cdot {\vgamma}^{\perp}(\vxi)
\end{bmatrix} = {\Qc(\vxi)}\Rc\vf.
      \end{equation}
      Therefore, $$\left[\Qc(\vxi)\right]^{-1}\frac{d}{ds} \Rc ({\Sc}\vf)(\vxi,s)= \Rc\vf(\vxi,s),$$
      which completes the proof of the theorem.
 \end{proof}
  \begin{cor}
  If $\Zc_{1}\cup\Zc_{2}$ is finite, then by applying $\Rc^{-1}$ to both sides of \eqref{eq: star inversion} we can recover $\vf(\vx)$ for all $\vx\in D_{1}$.
  \end{cor}
\subsection*{The cardinality of the sets $\Zc_{1}$ and $\Zc_{2}$}
It is obvious from the definition that the set $\Zc_1$ is finite. Let us investigate the cardinality of the set $\Zc_{2}$ by studying the cardinalities of its two parts. 

Recall from the proof of Lemma \ref{lem:zero sets of vgamma} that:
\begin{equation}
\vgamma^\dag(\vxi)  = - \sum_{i=1}^{m} \frac{c_{i}{\vgamma}_{i}\odot \vgamma_i^\perp }{\vxi \cdot {\vgamma}_{i}} 
= -\frac{\Pc^\dag(\vxi)}{\prod_{j=1}^{m} \vxi \cdot {\vgamma}_{j}} = -\frac{1}{\kappa(\vxi)}\left(\Pc^\dag_{11}(\vxi),\Pc^\dag_{12}(\vxi),\Pc^\dag_{22}(\vxi)\right),
\end{equation} 
where each $\Pc^\dag_{ij}(\vxi)$ is a homogeneous polynomial of degree $m-1$ in terms of  $\vxi=(\xi_{1},\xi_{2})$.
A vector $\vxi\in \Sb^1=\{(\xi_1,\xi_2): \xi_{1}^{2}+\xi_{2}^{2}-1=0\}$ is in the zero set of $\vgamma^\dag(\vxi)$ if and only if
$\Pc^\dag_{11}(\vxi)= \Pc^\dag_{12}(\vxi)= \Pc^\dag_{22}(\vxi)=0$. Using the classical B\'ezout's theorem with $\Pc_{c}(\vxi)= \xi_{1}^{2}+\xi_{2}^{2}-1$, we conclude that each $\Pc^\dag_{ij}(\vxi)$ either has finitely many zeros on $\Sb^{1}$, or is identically zero there (for more details, see \cite{Amb_Lat_star}). In a similar fashion, the set $\displaystyle \left\{\vxi:\sum_{i=1}^m\frac{c_i}{\vxi\cdot \vgamma_i} = 0\right\}$  is either finite or coincides with $\Sb^{1}\setminus \Zc_{1}$.  Hence, the set $\Zc_{2}$ is either finite or coincides with $\Sb^{1}\setminus \Zc_{1}$.

\vspace{1mm}

\begin{cor}\label{corr:invertibility of S}
The star transform is invertible if the set $\displaystyle \Sb^1\setminus\{\Zc_1\cup\Zc_2\}$ is non-empty.
\end{cor}

%\vspace{1mm}

\begin{defn}\label{def:symmetric star transform}
A star transform $\Sc$ is called  \textbf{symmetric} if $m=2k$ for some $k\in \Nb$ and (after possible index rearrangement)
$\vgamma_{i+k}= - \vgamma_{i}$ with $c_{i+k}=c_{i}$ for all $i=1,2,\dots ,k.$
\end{defn}

\begin{theorem}\label{th:star_non symmetric}
A star transform $\Sc$ is invertible if and only if it is not symmetric.
\end{theorem}
\begin{proof}
Suppose $\Sc$ is a symmetric star transform with $2k$ branches. Then $\Sc\vf$ contains less information than the standard (longitudinal, transverse, and mixed) integral transforms along straight lines limited to $k$ fixed directions, which is clearly not sufficient for the recovery of $\vf$. Thus, a symmetric star transform $\Sc$ is not invertible.  

To prove the statement in the other direction, assume that the star transform $\Sc$ with branches in the directions $\vgamma_{1},\vgamma_{2},\dots,\vgamma_{m}$ is not invertible. Then, by discussion preceding Corollary \ref{corr:invertibility of S}, for every $\vxi\in\mathbb{S}^1$ either $\Pc^\dag(\vxi)=0$ or $\sum_{i=1}^m c_i \left( \prod_{j\ne i}\vxi \cdot {\vgamma}_{j}\right)  = 0$. For $\vxi= \vgamma_{1}^{\perp}$, these two conditions can be rewritten as 
\begin{align*}
   c_{1}\prod_{j\ne 1}\vgamma_1^\perp \cdot {\vgamma}_{j}\left(\vgamma_{1}\odot \vgamma_1^\perp\right) = 0 \ \ \mbox{ or }\ \  c_{1}\prod_{j\ne 1}\vgamma_1^\perp \cdot {\vgamma}_{j} = 0.
\end{align*}
Since $c_1 \neq 0$, both of the above relations imply that
$$
\prod_{j\ne 1}\vgamma_1^\perp \cdot {\vgamma}_{j} = 0.
$$
Hence, $ \vgamma_{1}^{\perp}\cdot\vgamma_{j} =0$ for some $j =2,\dots,m$. Therefore, $ \vgamma_{1}=c\vgamma_{j}$ for some non-zero $c\in \Rb$. Furthermore, since $\vgamma_{1}$ is a unit vector and $\vgamma_{1}, \vgamma_{2},\dots ,\vgamma_{m}$ are distinct, we must have $\vgamma_{1}= -\vgamma_{j}$.  By repeating the same calculations with  $\vxi= \vgamma_{\ell}^{\perp}$ for all $\ell$, we  conclude that if the star transform is not invertible, then its rays have to point in pairwise opposite directions.

Without loss of generality, let us take $m = 2k$, $\vgamma_{i+k}= - \vgamma_{i}$, $i=1,\dots ,k$ and no pairs of unit vectors from  $\vgamma_{1},\vgamma_{2},\dots ,\vgamma_{k}$ are  collinear. To complete the proof, it remains to show that $c_{i+k} = c_{i}$ for all $i=1,2,\dots,k$. Using $\vgamma_{i+k}= - \vgamma_{i}$, $i=1, \dots,k$, we have 
$$\Pc^\dag(\vxi)=  (-1)^{k} \prod_{j=1}^{k}\left(\vxi \cdot {\vgamma}_{j}\right)\sum_{i=1}^{k}\left[\left(c_{i}-c_{i+k}\right) \prod_{\substack{j=1 \\ j\neq i}}^{k}\vxi \cdot {\vgamma}_{j}\right]\vgamma_{i}\odot\vgamma_i^\perp$$
and 
$$\sum_{i=1}^m c_i \left( \prod_{j\ne i}\vxi \cdot {\vgamma}_{j}\right) =  (-1)^{k} \prod_{j=1}^{k}\left(\vxi \cdot {\vgamma}_{j}\right)\sum_{i=1}^{k}\left(c_{i}-c_{i+k}\right) \left(\prod_{\substack{j=1 \\ j\neq i}}^{k}\vxi \cdot {\vgamma}_{j}\right).$$
As mentioned earlier, $\Sc$ is not invertible implies either $\Pc^\dag(\vxi)\equiv0$ or $\sum_{i=1}^m c_i \left( \prod_{j\ne i}\vxi \cdot {\vgamma}_{j}\right)  \equiv 0$. In both cases, we have 
$$\sum_{i=1}^{k}\left[\left(c_{i}-c_{i+k}\right) \prod_{\substack{j=1 \\ j\neq i}}^{k}\vxi \cdot {\vgamma}_{j}\right] \equiv 0.$$ 
For $\vxi= \vgamma_{1}^{\perp}$ this simplifies to 
$$ \left(c_{1}-c_{1+k}\right) \prod_{\substack{j=2 \\ j\neq 1}}^{k}\vgamma_{1}^{\perp} \cdot {\vgamma}_{j}=0.$$
Clearly, $\displaystyle \prod\limits_{\substack{j=2 \\ j\neq 1}}^{k}\vgamma_{1}^{\perp} \cdot {\vgamma}_{j}\neq 0$, otherwise $\vgamma_{j} = -\vgamma_{1}$ for some $j = 2,\dots ,k$, which is not possible. Therefore, we must have $c_{1} = c_{1+k}$. Repeating the procedure with $\vxi = \vgamma_{i}^{\perp}$ for $i = 2, \dots , k$ we get $c_{i} = c_{i+k}$ for all $ i= 1,2,\dots,k$. This shows that the star transform $\Sc$ under consideration  is symmetric. 
\end{proof}
\begin{cor}
A star transform $\Sc$ with an odd number of rays acting on symmetric $2$-tensor fields is invertible.
\end{cor}
\begin{rem}
When $m=2$ and $c_{1}=c_{2}$ then $\Sc\vf =(\Lc\vf, \Mc\vf, \Tc\vf)$. So, Theorem \ref{th:star_non symmetric} provides another approach to recovering a complete symmetric 2-tensor field from its longitudinal, transverse, and mixed V-line transforms. Note that, in this case, the matrix $\Qc(\vxi)$ is undefined if and only if $\vgamma_{1}=-\vgamma_{2}$, and the corresponding star transform is not invertible.
\end{rem}

%%%%%%%%%%%%%%%%%%%%%%%%%%%%%%%%%%%%%%%%%%%%%%%%%%%
\section{Additional remarks}\label{sec:add_remarks}
\begin{enumerate}
\item Certain facts established in this paper have notable differences with the results obtained previously in the similar problems for ray transforms on tensor fields and V-line transforms on vectors fields. In particular:
\begin{itemize}
\item In the case of ray transforms (integrating along straight lines), it is known that a symmetric $m$-tensor field $f$ can be recovered from the knowledge of its first $(m+1)$-integral moments. In the V-line setup, the results are identical for vector fields ($m = 1$) as reported in an earlier work \cite{Gaik_Mohammad_Rohit}, but fail for symmetric 2-tensor fields (see Lemma \ref{lem: recurrence relation for L} and the discussion below it). Moreover, in the case of transforms integrating symmetric 2-tensor fields along V-lines, we can write the $k$-th integral moments for $k\geq 2$ in terms of $0^{th}$ and $1^{st}$ order moments. Such a statement does not hold for the ray transforms of symmetric 2-tensor fields. 
\item It is well known that the kernel of the longitudinal (transverse) ray transform consists of potential (solenoidal) tensor fields. Similar to the discussion in the first point above, the kernel descriptions for V-line and straight-line setups are the same for vector fields ($m=1$). The difference appears again for symmetric $2$-tensor fields; the kernel of the longitudinal (transverse) VLT does not comprise the potential (solenoidal) tensor fields, as discussed in the Theorem \ref{th:kernel description}. 
\end{itemize}
\item The results obtained in this paper naturally lead to several possible directions of future research on the subject.
\begin{itemize}
    \item In this work we derived a number of exact formulas and algorithms for the reconstruction of a compactly supported, symmetric 2-tensor field from various combinations of its VLTs. The numerical implementation of these algorithms is an elaborate task, and the authors plan to address it in a future publication. 
    \item A natural extension of this work is to consider the corresponding problems for higher-order ($m>2$) tensor fields  in the 2-dimensional Euclidean space. 
    \item Another direction of future research is to study the higher dimensional analogs of these transforms. In the case of scalar fields, various conical transforms have been analyzed as generalization of the VLTs in $\mathbb{R}^n$, $n\ge3$ (e.g. see \cite{amb-lat_2019,  gouia2014analytical, Gouia_Amb_V-line}). Building up on these ideas, one can consider the problem of recovering a tensor field of order $m\ge2$ from various sets of properly defined conical transforms applied to that field. 
    The authors currently work on some of these questions and plan to address them in a separate publication.
\end{itemize}  
\end{enumerate}

%%%%%%%%%%%%%%%%%%%%%%%%%%%%%%%%%%%%%%%%%%%%%%%%%%%
\section{Acknowledgements}\label{sec:acknowledge}
GA was partially supported by the NIH grant U01-EB029826. RM was partially supported by SERB SRG grant No. SRG/2022/000947. IZ was supported by the Prime Minister's Research Fellowship from the Government of India.

%------------------------------------------
\bibliography{references}
\bibliographystyle{plain}

\end{document}